\documentclass[english]{article}
\usepackage[T1]{fontenc}
\usepackage[latin9]{inputenc}
\usepackage{geometry}
\usepackage{babel}
\usepackage{float}
\usepackage{units}
\usepackage{mathtools}
\usepackage{url}
\usepackage{bm}
\usepackage{amsmath}
\usepackage{amsthm}
\usepackage{amssymb}
\usepackage{graphicx}
\usepackage[authoryear]{natbib}
\PassOptionsToPackage{normalem}{ulem}
\usepackage{ulem}
\usepackage[unicode=true]
 {hyperref}

\makeatletter

\newcommand{\noun}[1]{\textsc{#1}}
\providecommand{\tabularnewline}{\\}
\floatstyle{ruled}
\newfloat{algorithm}{tbp}{loa}
\providecommand{\algorithmname}{Algorithm}
\floatname{algorithm}{\protect\algorithmname}

\theoremstyle{plain}
\newtheorem{thm}{\protect\theoremname}
\theoremstyle{plain}
\newtheorem{lem}[thm]{\protect\lemmaname}
\theoremstyle{definition}
\newtheorem{problem}[thm]{\protect\problemname}
\theoremstyle{plain}
\newtheorem{cor}[thm]{\protect\corollaryname}
\theoremstyle{plain}
\newtheorem{prop}[thm]{\protect\propositionname}

\usepackage{algorithmic}

\makeatother

\providecommand{\corollaryname}{Corollary}
\providecommand{\lemmaname}{Lemma}
\providecommand{\problemname}{Problem}
\providecommand{\propositionname}{Proposition}
\providecommand{\theoremname}{Theorem}

\begin{document}
\global\long\def\R{\mathbb{R}}%

\global\long\def\C{\mathbb{C}}%

\global\long\def\N{\mathbb{N}}%

\global\long\def\e{{\mathbf{e}}}%

\global\long\def\et#1{{\e(#1)}}%

\global\long\def\ef{{\mathbf{\et{\cdot}}}}%

\global\long\def\x{{\mathbf{x}}}%

\global\long\def\xt#1{{\x(#1)}}%

\global\long\def\xf{{\mathbf{\xt{\cdot}}}}%

\global\long\def\a{{\mathbf{a}}}%

\global\long\def\b{{\mathbf{b}}}%

\global\long\def\d{{\mathbf{d}}}%

\global\long\def\w{{\mathbf{w}}}%

\global\long\def\b{{\mathbf{b}}}%

\global\long\def\u{{\mathbf{u}}}%

\global\long\def\y{{\mathbf{y}}}%

\global\long\def\n{{\mathbf{n}}}%

\global\long\def\k{{\mathbf{k}}}%

\global\long\def\yt#1{{\y(#1)}}%

\global\long\def\yf{{\mathbf{\yt{\cdot}}}}%

\global\long\def\z{{\mathbf{z}}}%

\global\long\def\v{{\mathbf{v}}}%

\global\long\def\h{{\mathbf{h}}}%

\global\long\def\q{{\mathbf{q}}}%

\global\long\def\s{{\mathbf{s}}}%

\global\long\def\p{{\mathbf{p}}}%

\global\long\def\f{{\mathbf{f}}}%

\global\long\def\rb{{\mathbf{r}}}%

\global\long\def\rt#1{{\rb(#1)}}%

\global\long\def\rf{{\mathbf{\rt{\cdot}}}}%

\global\long\def\mat#1{{\ensuremath{\bm{\mathrm{#1}}}}}%

\global\long\def\vec#1{{\ensuremath{\bm{\mathrm{#1}}}}}%

\global\long\def\matN{\ensuremath{{\bm{\mathrm{N}}}}}%

\global\long\def\matX{\ensuremath{{\bm{\mathrm{X}}}}}%

\global\long\def\X{\ensuremath{{\bm{\mathrm{X}}}}}%

\global\long\def\matK{\ensuremath{{\bm{\mathrm{K}}}}}%

\global\long\def\K{\ensuremath{{\bm{\mathrm{K}}}}}%

\global\long\def\matA{\ensuremath{{\bm{\mathrm{A}}}}}%

\global\long\def\A{\ensuremath{{\bm{\mathrm{A}}}}}%

\global\long\def\matB{\ensuremath{{\bm{\mathrm{B}}}}}%

\global\long\def\B{\ensuremath{{\bm{\mathrm{B}}}}}%

\global\long\def\matC{\ensuremath{{\bm{\mathrm{C}}}}}%

\global\long\def\C{\ensuremath{{\bm{\mathrm{C}}}}}%

\global\long\def\matD{\ensuremath{{\bm{\mathrm{D}}}}}%

\global\long\def\D{\ensuremath{{\bm{\mathrm{D}}}}}%

\global\long\def\matE{\ensuremath{{\bm{\mathrm{E}}}}}%

\global\long\def\E{\ensuremath{{\bm{\mathrm{E}}}}}%

\global\long\def\matF{\ensuremath{{\bm{\mathrm{F}}}}}%

\global\long\def\F{\ensuremath{{\bm{\mathrm{F}}}}}%

\global\long\def\matH{\ensuremath{{\bm{\mathrm{H}}}}}%

\global\long\def\H{\ensuremath{{\bm{\mathrm{H}}}}}%

\global\long\def\matP{\ensuremath{{\bm{\mathrm{P}}}}}%

\global\long\def\matG{\ensuremath{{\bm{\mathrm{G}}}}}%

\global\long\def\P{\ensuremath{{\bm{\mathrm{P}}}}}%

\global\long\def\matU{\ensuremath{{\bm{\mathrm{U}}}}}%

\global\long\def\matV{\ensuremath{{\bm{\mathrm{V}}}}}%

\global\long\def\V{\ensuremath{{\bm{\mathrm{V}}}}}%

\global\long\def\matW{\ensuremath{{\bm{\mathrm{W}}}}}%

\global\long\def\matM{\ensuremath{{\bm{\mathrm{M}}}}}%

\global\long\def\M{\ensuremath{{\bm{\mathrm{M}}}}}%

\global\long\def\calA{{\cal A}}%

\global\long\def\calE{{\cal E}}%

\global\long\def\calF{{\cal F}}%

\global\long\def\calK{{\cal K}}%

\global\long\def\calH{{\cal H}}%

\global\long\def\calY{{\cal Y}}%

\global\long\def\calP{{\cal P}}%

\global\long\def\calX{{\cal X}}%

\global\long\def\calS{{\cal S}}%

\global\long\def\calT{{\cal T}}%

\global\long\def\Normal{{\cal \mathcal{N}}}%

\global\long\def\GP{{\cal \mathcal{GP}}}%

\global\long\def\matQ{{\mat Q}}%

\global\long\def\Q{{\mat Q}}%

\global\long\def\matR{\mat R}%

\global\long\def\matS{\mat S}%

\global\long\def\matY{\mat Y}%

\global\long\def\matI{\mat I}%

\global\long\def\I{\mat I}%

\global\long\def\matJ{\mat J}%

\global\long\def\matZ{\mat Z}%

\global\long\def\Z{\mat Z}%

\global\long\def\matW{{\mat W}}%

\global\long\def\W{{\mat W}}%

\global\long\def\matL{\mat L}%

\global\long\def\S#1{{\mathbb{S}_{N}[#1]}}%

\global\long\def\IS#1{{\mathbb{S}_{N}^{-1!}[#1]}}%

\global\long\def\PN{\mathbb{P}_{N}}%

\global\long\def\TNormS#1{\|#1\|_{2}^{2}}%

\global\long\def\ITNormS#1{\|#1\|_{2}^{-2}}%

\global\long\def\ONorm#1{\|#1\Vert_{1}}%

\global\long\def\TNorm#1{\|#1\|_{2}}%

\global\long\def\InfNorm#1{\|#1\|_{\infty}}%

\global\long\def\FNorm#1{\|#1\|_{F}}%

\global\long\def\FNormS#1{\|#1\|_{F}^{2}}%

\global\long\def\UNorm#1{\|#1\|_{\matU}}%

\global\long\def\UNormS#1{\|#1\|_{\matU}^{2}}%

\global\long\def\UINormS#1{\|#1\|_{\matU^{-1}}^{2}}%

\global\long\def\ANorm#1{\|#1\|_{\matA}}%

\global\long\def\BNorm#1{\|#1\|_{\mat B}}%

\global\long\def\ANormS#1{\|#1\|_{\matA}^{2}}%

\global\long\def\AINormS#1{\|#1\|_{\matA^{-1}}^{2}}%

\global\long\def\WNorm#1{\|#1\|_{\matW}}%

\global\long\def\T{\textsc{T}}%

\global\long\def\conj{\textsc{*}}%

\global\long\def\pinv{\textsc{+}}%

\global\long\def\Prob{\operatorname{Pr}}%

\global\long\def\Expect{\operatorname{\mathbb{E}}}%

\global\long\def\ExpectC#1#2{{\mathbb{E}}_{#1}\left[#2\right]}%

\global\long\def\VarC#1#2{{\mathbb{\mathrm{Var}}}_{#1}\left[#2\right]}%

\global\long\def\dotprod#1#2#3{(#1,#2)_{#3}}%

\global\long\def\dotprodN#1#2{(#1,#2)_{{\cal N}}}%

\global\long\def\dotprodH#1#2{\langle#1,#2\rangle_{{\cal {\cal H}}}}%

\global\long\def\dotprodsqr#1#2#3{(#1,#2)_{#3}^{2}}%

\global\long\def\Trace#1{{\bf Tr}\left(#1\right)}%

\global\long\def\Vec#1{{\bf Vec}\left(#1\right)}%

\global\long\def\nnz#1{{\bf nnz}\left(#1\right)}%

\global\long\def\MSE#1{{\bf MSE}\left(#1\right)}%

\global\long\def\WMSE#1{{\bf WMSE}\left(#1\right)}%

\global\long\def\EWMSE#1{{\bf EWMSE}\left(#1\right)}%

\global\long\def\nicehalf{\nicefrac{1}{2}}%

\global\long\def\nicebetahalf{\nicefrac{\beta}{2}}%

\global\long\def\argmin{\operatornamewithlimits{argmin}}%

\global\long\def\argmax{\operatornamewithlimits{argmax}}%

\global\long\def\norm#1{\Vert#1\Vert}%

\global\long\def\sign{\operatorname{sign}}%

\global\long\def\proj{\operatorname{proj}}%

\global\long\def\diag{\operatorname{diag}}%

\global\long\def\VOPT{\operatorname\{VOPT\}}%

\global\long\def\dist{\operatorname{dist}}%

\global\long\def\diag{\operatorname{diag}}%

\global\long\def\supp{\operatorname{supp}}%

\global\long\def\sp{\operatorname{span}}%

\global\long\def\rank{\operatorname{rank}}%

\global\long\def\abs{\operatorname{abs}}%

\global\long\def\onehot{\operatorname{onehot}}%

\global\long\def\softmax{\operatorname{softmax}}%

\newcommand*\diff{\mathop{}\!\mathrm{d}} 

\global\long\def\dd{\diff}%

\global\long\def\whatlambda{\w_{\lambda}}%

\global\long\def\Plambda{\mat P_{\lambda}}%

\global\long\def\Pperplambda{\left(\mat I-\Plambda\right)}%

\global\long\def\Mlambda{\matM_{\lambda}}%

\global\long\def\Mlambdafull{\matM+\lambda\matI}%

\global\long\def\Mlambdafullinv{\left(\matM+\lambda\matI\right)^{-1}}%

\global\long\def\Mdaggerlambda{{\mathbf{\mat M_{\lambda}^{+}}}}%

\global\long\def\Xdaggerlambda{{\mathbf{\mat X_{\lambda}^{+}}}}%

\global\long\def\XT{{\mathbf{X}^{\T}}}%

\global\long\def\XXT{{\matX\mat X^{\T}}}%

\global\long\def\XTX{{\matX^{\T}\mat X}}%

\global\long\def\VT{{\mathbf{V}^{\T}}}%

\global\long\def\VVT{{\matV\mat V^{\T}}}%

\global\long\def\VTV{{\matV^{\T}\mat V}}%

\global\long\def\varphibar#1#2{{\bar{\varphi}_{#1,#2}}}%

\global\long\def\varphilambda{{\varphi_{\lambda}}}%

\global\long\def\RE#1{{\bf Re}\left(#1\right)}%

\title{Low-Rank Updates of Matrix Square Roots}
\author{Shany Shmueli\\
 Tel Aviv University \\
 shanyshmueli@mail.tau.ac.il\\
\and Petros Drineas\\
Purdue University\\
pdrineas@purdue.edu\\
\and Haim Avron \\
 Tel Aviv University \\
haimav@tauex.tau.ac.il}
\maketitle
\begin{abstract}
Models in which the covariance matrix has the structure of a sparse
matrix plus a low rank perturbation are ubiquitous in data science
applications. It is often desirable for algorithms to take advantage
of such structures, avoiding costly matrix computations that often
require cubic time and quadratic storage. This is often accomplished
by performing operations that maintain such structures, e.g. matrix
inversion via the Sherman-Morrison-Woodbury formula. In this paper
we consider the matrix square root and inverse square root operations.
Given a low rank perturbation to a matrix, we argue that a low-rank
approximate correction to the (inverse) square root exists. We do
so by establishing a geometric decay bound on the true correction's
eigenvalues. We then proceed to frame the correction as the solution
of an algebraic Riccati equation, and discuss how a low-rank solution
to that equation can be computed. We analyze the approximation error
incurred when approximately solving the algebraic Riccati equation,
providing spectral and Frobenius norm forward and backward error bounds.
Finally, we describe several applications of our algorithms, and demonstrate
their utility in numerical experiments.
\end{abstract}

\section{Introduction}

In applications, and in particular data science applications, one
often encounters matrices that are low-rank perturbations of another
(perhaps simpler) matrix. For example, models in which the covariance
matrix has the structure of a sparse matrix plus a low rank perturbation
are common. In another example, it is common for algorithms to maintain
a matrix that is iteratively updated by low-rank perturbations.

It is often desirable for algorithms to take advantage of such structures,
avoiding costly matrix computations that often require cubic time
and quadratic storage. An indispensable tool for utilizing low-rank
perturbations is the famous Sherman-Morrison-Woodbury formula, which
shows that the inverse of a low-rank perturbation can be obtained
using a low-rank correction of the inverse. While the usefulness of
the Sherman-Morrison-Woodbury formula cannot be understated, other
matrix functions also frequently appear in applications. One naturally
asks the following questions. Given a matrix function $f$, when does
a low rank perturbation of a matrix $\matA$ correspond to a (approximately)
low-rank correction of $f(\matA)$? Can we find a high quality approximate
correction efficiently, i.e. without computing the exact correction
and truncating it using a SVD?

In this paper, we answer these questions affirmatively for two important
and closely related functions: the matrix square root and inverse
square root. The square root operation on matrices is a fundamental
operation that frequently appears in mathematical analysis. Moreover,
matrix square roots and their inverse arise frequently in data science
applications, e.g. when sampling from a high dimensional multivariate
Gaussian distribution, or when whitening data. Computing the square
root (or inverse square root) of low-rank perturbations of simple
matrices (e.g. diagonal matrices) appear in quite a few data science
applications, e.g. in the aforementioned applications when the covariance
matrix follows a spiked population model. For more discussion on applications
of low-rank perturbations of matrix square root, see Section~\ref{sec:applications}.

In particular, we show that given a low rank perturbation $\matD$
to a matrix $\matA$, we can approximate $(\matA+\matD)^{\nicehalf}$
well by a low rank correction to $\matA^{\nicehalf}$. We do so by
proving a geometric decay bound on the eigenvalues of $(\matA+\matD)^{\nicehalf}-\matA^{\nicehalf}$.
We also provide a similar bound for the inverse square root. We then
proceed to show that the exact update is a solution of an algebraic
Riccati equation, and discuss how a low-rank approximate solution
to that equation can be computed. This allows us to propose concrete
algorithms for updating and downdating the matrix square root and
matrix inverse square root. Finally, we report experiments that corroborate
our theoretical results.

\subsection{Related Work}

Most previous work on the matrix square root focused on computing
$\matA^{\nicehalf}\x$ and $\matA^{-\nicehalf}\x$ for a given vector
$\x$ using a Krylov method, possibly with preconditioning~\citep{Higham08,AEP13,ChowSaad14,FGS14}.
The motivation for most of the aforementioned works is sampling from
a multivariate Gaussian distributions. Worth mentioning is recent
work by \citet{PleissEtAl20} which combines a Krylov subspace method
with a rational approximation, and also allows preconditioning.

The problem of updating a function of a matrix after a low rank perturbation,
i.e. computing $f(\matA+\matD)$ given $f(\matA)$ where $\matD$
is low-rank, has been recently receiving attention. The Sherman-Morrison-Woodbury
formula is a well known formula for updating the matrix inverse, i.e.
$f(x)=x^{-1}$. \citet{BV00} showed that that when $f$ is a rational
function of degree $q$, a rank one perturbation of $\matA$ corresponds
to a rank $q$ perturbation of $f(\matA)$. This paper also provides
an explicit formula for the low rank perturbation. \citet{Higham08}
showed that a rank $k$ perturbation of $\matA=\alpha\matI$ corresponds
to a rank $k$ perturbation of $f(\matA)$ for \emph{any }$f$ (see
Theorem 1.35 therein). As for inexact corrections, \citet{BKS18}
proposed a Krylov method for computing a low-rank correction of $f(\matA)$
that approximates $f(\matA+\matD)$ well for any analytic $f$ (however,
approximation quality depends on properties of the function itself,
e.g. how well it is approximated by a polynomial). In followup work,
they proposed a rational Krylov method \citep{BCKS20}, citing the
matrix square root as an example of a case in which their original
method might have slow convergence.

The work most similar to ours is \citep{FHL22}. In that paper, the
authors consider the problem of computing the square root of a matrix
of the form $\alpha\matI+\matU\matV^{\T}$, as a correction of $\sqrt{\alpha}\matI$.
Due to \citep[Theorem 1.35]{Higham08}, the rank of the correction
is the same as the rank of $\matU\matV^{\T}$. However, the formula
in \citep[Theorem 1.35]{Higham08} requires $\matV^{\T}\matU$ to
be non-singular, which is not required for the square root to be defined.
The authors circumvent this issue by suggesting another formula for
the square root, or by using a Newton iteration. In a way, the algorithm
in \citep{FHL22} is more general than our method since it allows
non-symmetric updates. However, in another way it is less general:
the matrix to be perturbed must be a scaled identity matrix. Moreover,
we also suggest a method for updating the inverse square root.

\section{Preliminaries}

\subsection{Notation and Basic Definitions}

We denote scalars using Greek letters or using $x,y,\dots$. Vectors
are denoted by $\x,\y,\dots$ and matrices by $\matA,\mat B,\dots$.
The $n\times n$ identity matrix is denoted $\matI_{n}$. We use the
convention that vectors are column-vectors.

Given a symmetric positive-semidefinite matrix $\matA\in\R^{n\times n}$,
another matrix $\matB\in\R^{n\times n}$ is a \emph{square root} of
$\matA$ if $\matB^{2}=\matA$. There is a unique square root of $\matA$
that is also positive semi-definite, which is called the \emph{principal
square root }and we denote it by $\matA^{\nicehalf}$.

Given two matrices $\matF$ and $\matG$, the $(\matF,\matG)$-displacement
rank of $\matA$ is defined as the rank of $\matF\matA-\matA\mat G$.
Displacement structures and displacement rank are closely connected
to the Sylvester equation.

\subsection{Low-Rank Algebraic Riccati Equation}

Consider the following equation in $\matX\in\R^{n\times n}$,
\begin{equation}
\matE\matX+\matX\matE+\alpha\matX^{2}=\matG^{\T}\matG\label{eq:algebraic-ricatti}
\end{equation}
where $\matE\in\R^{n\times n}$ is a symmetric full-rank matrix, $\matG\in\R^{k\times n}$
where $k\ll n$, and $\alpha=\pm1$. Our algorithms are based on approximately
solving Eq.~(\ref{eq:algebraic-ricatti}) with a positive semidefinite
low-rank $\matX$.

If $\alpha=+1$, Eq.~(\ref{eq:algebraic-ricatti}) is an instance
of the \emph{algebraic Riccati equation}. There is a rich literature
on algorithms for finding low rank solutions for the algebraic Riccati
equation. We note the survey due to \citet{BS13}, and the book by
\citet{BIM11}. In our experiment, we use a recently  proposed meta-scheme
for approximately solving a slightly more general version of Eq.~(\ref{eq:algebraic-ricatti})
based on Riemannian optimization \citep{BV14}. When applied to Eq.~(\ref{eq:algebraic-ricatti})
their scheme assumes the ability to take products of $\matE$ by a
vector, and to solve linear equations where the matrix is equal to
$\matE^{2}+\textrm{low-rank}$, which is easily achievable via the
Sherman-Morrison-Woodbury formula if we have access to an oracle that
multiplies $\matE^{-1}$ by a vector. Under the assumption that each
rank update in \citep{BV14} requires $O(1)$ trust-region iterations,
and that the target maximum rank of $\matX$ is $r$, the overall
cost of the scheme in \citep{BV14} is $O((T_{\matE}+T_{\matE^{-1}})r^{2}+nr^{4})$
where $T_{\matE}$ and $T_{\matE^{-1}}$ is the cost of multiplying
$\matE$ and $\matE^{-1}$ by a vector (respectively). In most of
our applications $\matE$ is diagonal, so the cost reduces to $O(nr^{4})$.

The method described in \citep{BV14} handles only the case of $\alpha=+1$.
However, it can be generalized to the case of $\alpha=-1$. We give
details in Appendix~\ref{app:negative-ricatti}.

In our algorithms, we denote the process of solving Eq.~(\ref{eq:algebraic-ricatti})
via the notation
\[
\matU\gets\textrm{RiccatiLRSolver}(\matE,\matG,\alpha,r)
\]
where $r$ is the target maximum rank, and $\matU\in\R^{n\times r}$
is a symmetric factor of the solution $\matX$ (i.e., $\matX=\matU\matU^{\T}$).
In the complexity analyses we assume that this process takes $O((T_{\matE}+T_{\matE^{-1}})r^{2}+nr^{4})$,
as justified by the discussion above. However, we stress that our
algorithm can use any algorithm for finding low-rank approximate solution
to the algebraic Ricatti equation. We remark that our algorithm actually
usse only the case $\alpha=+1$, but for some discussions it is useful
to consider also the ability to solve for $\alpha=-1$. 

\subsection{Problem Statement}

Let $\matA\in\R^{n\times n}$ be a a symmetric positive semidefinite
matrix. Suppose we are given $\matA^{\nicehalf}$, perhaps implicitly
(i.e., as a function that maps a vector $\x$ to $\matA^{\nicehalf}\x$).
Given a perturbation $\matD\in\R^{n\times n}$ of rank $k\ll n$,
our goal is to approximate $(\matA+\matD)^{\nicehalf}$ using a low-rank
correction of $\matA^{\nicehalf}$. That is, to find a $\tilde{\Delta}$
of rank $r$ such that $(\matA+\matD)^{\nicehalf}\approx\matA^{\nicehalf}+\tilde{\Delta}$.
The rank $r$ of $\tilde{\Delta}$ should be treated as a parameter,
and should optimally be $O(k)$. We show in Section~\ref{sec:decay}
that we can expect to find a good approximation with $r\ll n$.

We make two additional assumptions on $\matD$. First, we assume that
it is either positive semidefinite or negative semidefinite. Indefinite
perturbations can be handled by splitting the update into two semidefinite
perturbations and applying our algorithms sequentially. Secondly,
we assume that $\matD$ is given in a symmetric factorized form. We
can combine the last two assumptions in a single assumption by assuming
we are given a $\matZ\in\R^{n\times k}$ such that $\matD=\alpha\matZ\matZ^{\T}$
where $\alpha=\pm1$. We refer to the case of $\alpha=+1$ as \emph{updating
the square root}, and $\alpha=-1$ as \emph{downdating the square
root.}

In the case of updating the square root, we are guaranteed that $\matA+\matZ\matZ^{\T}$
is positive definite for any $\matZ$, but this does not necessarily
holds for downdating. Thus, for downdates we further assume that $\matA-\matZ\matZ^{\T}$
is positive definite. The following Lemma gives an easy way to test
this condition in cases we also have access to the inverse of $\matA^{\nicehalf}$
or $\matA$.
\begin{lem}
\label{lem:downdate-psd}Suppose that $\matA\in\R^{n\times n}$ is
symmetric positive definite, and $\matZ\in\R^{n\times k}$ for $k\leq n$.
Then $\matA-\matZ\matZ^{\T}$ is positive semidefinite if and only
if $\matI_{k}-\matZ^{\T}\matA^{-1}\matZ$ is positive semi definite.
\end{lem}

\begin{proof}
$\matA-\matZ\matZ^{\T}\succeq0$ if and only if $\matA\succeq\matZ\matZ^{\T}$.
Multiplying by $\matA^{-\nicehalf}$ on both sides, we see that this
holds if and only if $\matA^{-\nicehalf}\matZ\matZ^{\T}\matA^{-\nicehalf}\preceq\matI_{n}$.
The last inequality holds if and only if all the eigenvalues of $\matA^{-\nicehalf}\matZ\matZ^{\T}\matA^{-\nicehalf}$
are smaller or equal to $1$. However, the non-zero eigenvalues of
that matrix are equal to the eigenvalues of $\matZ^{\T}\matA^{-1}\matZ$,
so $\matA-\matZ\matZ^{\T}\succeq0$ if and only if $\matZ^{\T}\matA^{-1}\matZ\preceq\matI_{k}$,
which is equivalent to $\matI_{k}-\matZ^{\T}\matA^{-1}\matZ$ being
positive definite.
\end{proof}
The correction $\tilde{\Delta}$ should be returned in factorized
form. For reasons that will become apparent in our algorithm, the
correction will be positive semidefinite when $\matD$ is positive
semidefinite, and negative semidefinite when $\matD$ is negative
semidefinite. Thus, we require our algorithm to return a $\matU\in\R^{n\times r}$
such that $(\matA+\alpha\matD)^{\nicehalf}\approx\matA^{\nicehalf}+\alpha\matU\matU^{\T}$.

The discussion so far was for updating or downdating the square root.
We also aim at correcting the inverse of the square root, i.e. $\matA^{-\nicehalf}$.
Again, we will require $\matD$ to be either positive semidefinite
or negative semidefinite, and $\tilde{\Delta}$ to be factorized as
well and definite. However, for updating the inverse of the square
root, the definiteness of $\tilde{\Delta}$ will be opposite to the
one of $\matD$.

We can capture the distinction between updating the square root and
the inverse square root with an additional parameter $\beta=\pm1$,
where we wish to update $\matA^{\nicefrac{\beta}{2}}$. Putting it
all together, we arrive at the following problem:
\begin{problem}
\label{prob:main}Given implicit access to $\matA^{\nicehalf}$ and/or
$\matA^{-\nicehalf}$, $\matZ\in\R^{n\times k}$, $\alpha=\pm1$,
$\beta=\pm1$ and target rank $r$, return a $\matU\in\R^{n\times r}$
such that 
\[
(\matA+\alpha\matZ\matZ^{\T})^{\nicefrac{\beta}{2}}\approx\matA^{\beta/2}+\alpha\beta\matU\matU^{\T}.
\]
\end{problem}

\section{\label{sec:decay}Decay Bounds for Square Roots Corrections}

Given a matrix $\matA$ and a low-rank perturbation $\matD$, our
goal is to find a low-rank correction $\tilde{\Delta}$ to $\matA^{\beta/2}$.
However, one can ask whether such a correction even exists? Let $\Delta$
denote the exact correction, i.e. 
\[
\Delta\coloneqq\left(\matA+\matD\right)^{\nicefrac{\beta}{2}}-\matA^{\nicefrac{\beta}{2}}.
\]
In this section we show that the eigenvalues of $\Delta$ exhibit
a geometric decay. Thus, by applying eigenvalue thresholding to $\Delta$
we can obtain a low-rank approximate correction $\tilde{\Delta}$
(however, our algorithm uses a different method for finding $\tilde{\Delta}$).

\subsection{Decay Bound for Square Root Corrections}

We first consider the case that $\beta=1$, so we are perturbing the
square root.

Let us denote $\matB\coloneqq\matA+\matD$. The following is a known
identity: 
\[
\matA^{\nicehalf}\Delta+\Delta\matB^{\nicehalf}=\matB-\matA.
\]
Indeed, since $\matB=(\matB^{\nicehalf})^{2}=(\matA^{\nicehalf}+\Delta)^{2}$
we have 
\begin{eqnarray*}
\matB-\matA & = & (\matA^{\nicehalf}+\Delta)^{2}-\matA\\
 & = & \matA+\matA^{\nicehalf}\Delta+\Delta(\matA^{\nicehalf}+\Delta)-\matA\\
 & = & \matA^{\nicehalf}\Delta+\Delta(\matA^{\nicehalf}+\Delta)\\
 & = & \matA^{\nicehalf}\Delta+\Delta\matB^{\nicehalf}
\end{eqnarray*}

In our case, that is when $\matB=\matA+\matD$, we see that $\Delta$
upholds the following Sylvester equation:
\begin{equation}
\matA^{\nicehalf}\Delta+\Delta\matB^{\nicehalf}=\matD\label{eq:sylvester_sqrt}
\end{equation}
Since the rank of $\matD$ is $k$, we see that $\Delta$ has a $(\matA^{\nicehalf},-\matB^{\nicehalf})$-displacement
rank of $k$. \citet{BT17} recently developed singular value decay
bounds for matrices with displacement structure. Using their results,
we can prove the following bound.
\begin{thm}
\label{prop:decay}Suppose that both $\matA$ and $\matB=\matA+\matD$
are symmetric positive definite matrices, and that $\matD$ is of
rank $k$. Let $\Delta=\matB^{\nicehalf}-\matA^{\nicehalf}$. Then
for $j\geq1$, the singular values of $\Delta$ satisfy the following
bound
\[
\sigma_{j+kl}(\Delta)\leq4\left[\exp\left(\frac{\pi^{2}}{2\log(4\hat{\kappa})}\right)\right]^{-2l}\sigma_{j}(\Delta)
\]
where
\[
\hat{\kappa}=\frac{2(\sqrt{\TNorm{\matA}+\TNorm{\matD}}+\sqrt{\lambda_{\min}(\matA)}/2)}{\sqrt{\lambda_{\min}(\matA)}}
\]
\end{thm}

\begin{proof}
Let $\mu=\sqrt{\lambda_{\min}(\matA)}$ and $\delta=\sqrt{\TNorm{\matA}+\TNorm{\matD}}$.
The matrix $\Delta$ upholds the following Sylvester equation 
\[
(\matA^{\nicehalf}-(\mu/2)\matI_{n})\Delta+\Delta(\matB^{\nicehalf}+(\mu/2)\matI_{n})=\matD
\]
Thus, the results in~\citep{BT17} show that we can bound 
\begin{equation}
\sigma_{j+kl}(\Delta)\leq Z_{l}(E,F)\sigma_{j}(\Delta)\label{eq:bt-1}
\end{equation}
where $E$ is any set that contains the spectrum of $\matA^{\nicehalf}-(\mu/2)\matI_{n}$,
$F$ is any set that contains the spectrum of $-(\matB^{\nicehalf}+(\mu/2)\matI_{n})$,
and $Z_{l}(E,F)$ is the Zolotarev number.

Let $a=\mu/2$. Since $\mu$ is the minimal eigenvalue of $\matA^{\nicehalf}$,
all the eigenvalues of $\matA^{\nicehalf}-(\mu/2)\matI_{n}$ are bigger
than $a$ or equal to it. Since $\matB$ is by assumption positive
definite, all the eigenvalues of $-(\matB^{\nicehalf}+(\mu/2)\matI_{n})$
are smaller than $-a$ or equal to it. Let $b=\delta+\mu/2$. Obviously,
all the eigenvalues of $\matA^{\nicehalf}-(\mu/2)\matI_{n}$ are smaller
than $b$. Furthermore, since $\TNorm{\matB}\leq\TNorm{\matA}+\TNorm{\matD}$,
all the eigenvalues of $-(\matB^{\nicehalf}+(\mu/2)\matI_{n})$ are
bigger than or equal to $-b$. Thus, we can take $E=[a,b]$ and $F=[-b,-a]$.
In~\citep{BT17} it is also shown that 
\[
Z_{l}([a,b],[-b,-a])\leq4\left[\exp\left(\frac{\pi^{2}}{2\log(4b/a)}\right)\right]^{-2l}
\]
plugging that into Eq.~(\ref{eq:bt-1}) gives the desired bound.
\end{proof}
\begin{figure}
\begin{centering}
\begin{tabular}{cc}
\includegraphics[width=0.4\columnwidth]{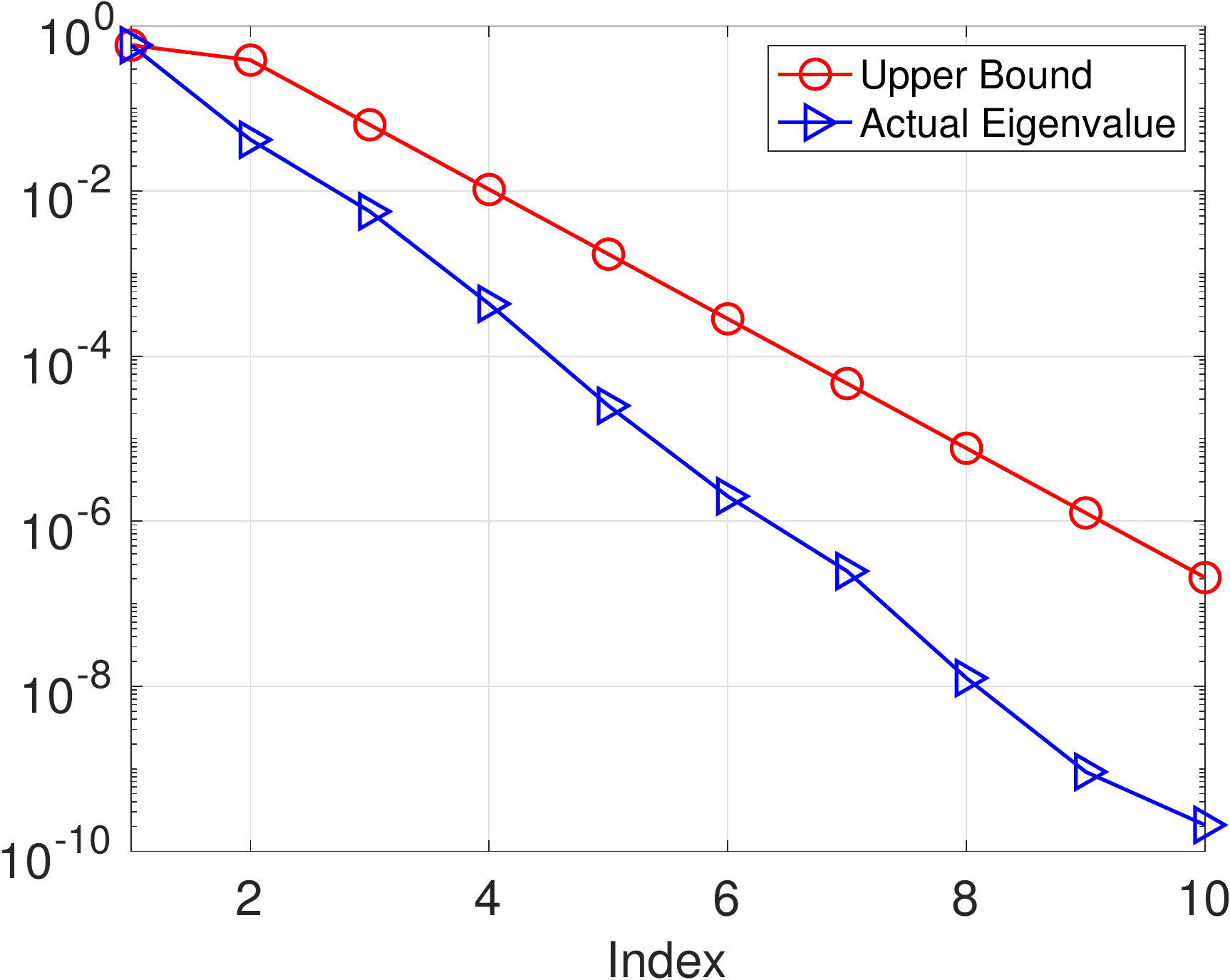} & \includegraphics[width=0.4\columnwidth]{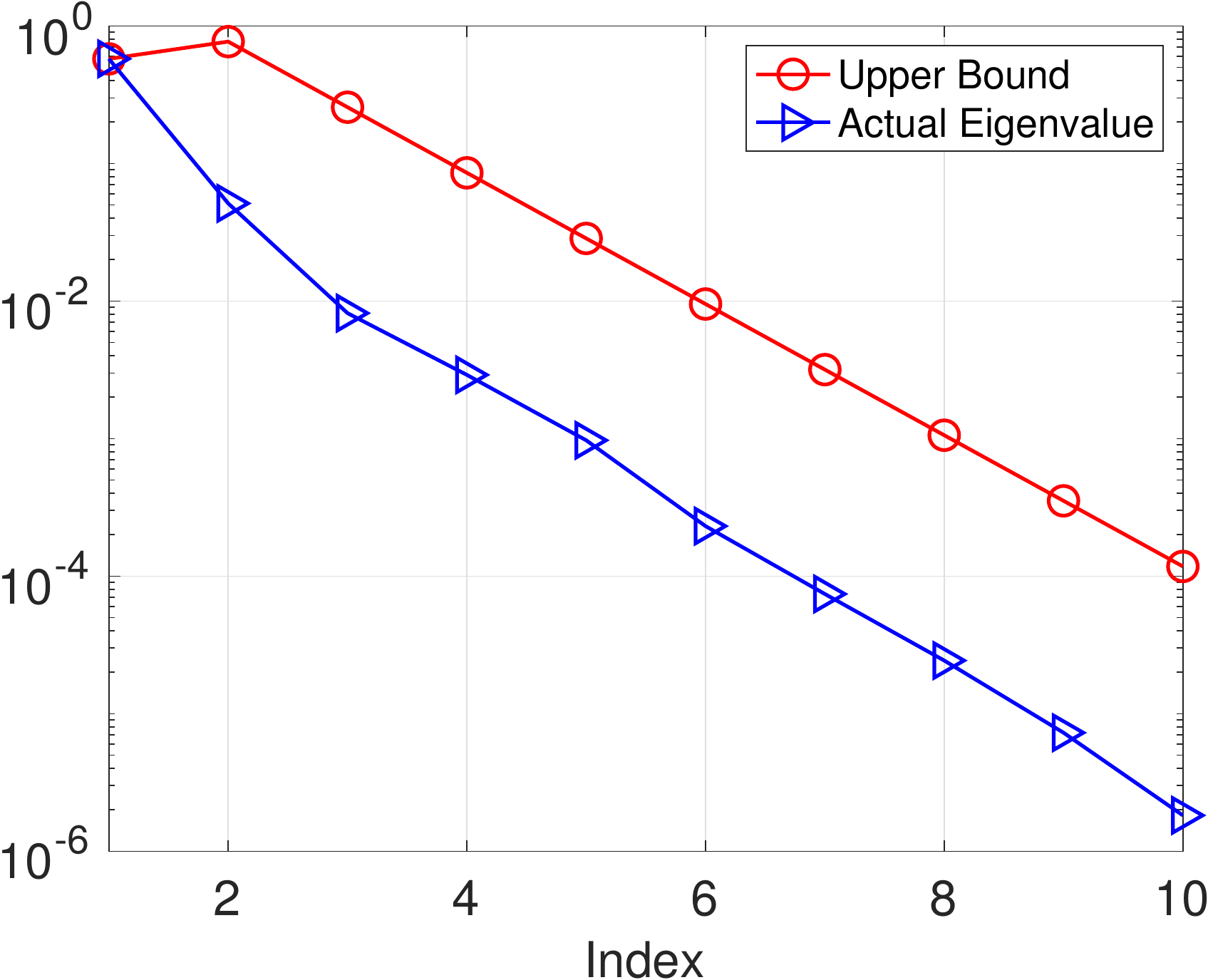}\tabularnewline
\end{tabular}
\par\end{centering}
\caption{\label{fig:bound-vs-reality}Illustration of the eigenvalue decay
bound of Theorem~\ref{prop:decay} vs. the actual decay observed
on two simple examples - on the left we sampled A entries uniformly
and on the right we used logspace sampling.}
\end{figure}

The theorem bounds the singular values. However, if $\matD$ is positive
semidefinite, then $\Delta$ is also positive semidefinite, and the
bound is actually on the eigenvalues. Figure~\ref{fig:bound-vs-reality}
illustrates the bound vs. actual decay of the eigenvalues on two simple
test cases. In both examples, $\matA\in\R^{100\times100}$ is a diagonal
matrix. In the left graph, the diagonal entries are sampled uniformly
from $U(0,1)$. In the right graph, diagonal entries are logarithmically
spaced between $10^{-3}$ and $10^{3}$. The perturbation is $\matD=\z\z^{\T}$,
where $\z$ is a normalized Gaussian vector.

We remark the previous to the aforementioned theoretical results of
~\citet{BT17}, it has been empirically observed that if the righthand
side of a Sylvester equation is low rank, then the solution is well
approximated using a low rank matrix~\citep{BK14}.

\subsection{Decay Bound for Inverse Square Root Corrections}

Observe that 
\begin{equation}
-\matA^{-1}\matD\matB^{-1}=\matB^{-1}-\matA^{-1}=\matA^{-\nicehalf}\Delta+\Delta\matB^{-\nicehalf}.\label{eq:sylvester_sqrt_inv}
\end{equation}
Since $\matD$ is rank $k$, so the matrix $-\matA^{-1}\matD\matB^{-1}$
is of rank at most $k$. Thus, similarly to Theorem~\ref{prop:decay},
by observing that $\TNorm{\matB^{-1}}\leq\TNorm{\matA^{-1}}\left(1+\TNorm{\matD}\TNorm{\matB^{-1}}\right)$
(which follows from $\matB^{-1}=\matA^{-1}-\matB^{-1}\matD\matA^{-1}$),
we can prove the following bound on the singular values of $\Delta$:
\[
\sigma_{j+kl}(\Delta)\leq4\left[\exp\left(\frac{\pi^{2}}{2\log(4\hat{\kappa})}\right)\right]^{-2l}\sigma_{j}(\Delta)
\]
where
\[
\hat{\kappa}=\frac{2(\sqrt{\lambda_{\min}(\matA)^{-1}(1+\TNorm{\matD}\lambda_{\min}(\matB)^{-1})}+\sqrt{\lambda_{\max}(\matA)^{-1}}/2)}{\sqrt{\lambda_{\max}(\matA)^{-1}}}
\]
We omit the proof since it is almost identical to the proof of Theorem~\ref{prop:decay}.

\section{Equation for Square Roots Corrections and Error Analysis}

Our algorithms are based on writing $\Delta$ as a solution of an
equation, and then finding a low-rank approximate solution $\tilde{\Delta}$.
Seemingly, Eqs.~(\ref{eq:sylvester_sqrt}) and (\ref{eq:sylvester_sqrt_inv})
are the equations we need. However, these equations contain the unknown
$\matB^{\nicebetahalf}$ so they are not useful for us algorithmically.
We derive a different equation instead. In particular, we write $\Delta$
as the solution of an algebraic Riccati equation, i.e. in the form
of Eq.~(\ref{eq:algebraic-ricatti}).

We can combine Eqs.~(\ref{eq:sylvester_sqrt}) and (\ref{eq:sylvester_sqrt_inv})
into a single equation:
\[
\matA^{\nicebetahalf}\Delta+\Delta\matB^{\nicebetahalf}=\matB^{\beta}-\matA^{\beta}
\]
Recalling that $\matB^{\nicebetahalf}=\matA^{\nicebetahalf}+\Delta$,
and plugging it into the last equation we get
\begin{equation}
\matA^{\nicebetahalf}\Delta+\Delta\matA^{\nicebetahalf}+\Delta^{2}=\matB^{\beta}-\matA^{\beta}\label{eq:pre-ricatti}
\end{equation}

In order for the equation to fit Eq.~(\ref{eq:algebraic-ricatti})
we must write the right side as a positive semi-definite factorized
matrix. The first step is finding a matrix $\matV\in\R^{n\times k}$
such that 
\[
\matB^{\beta}-\matA^{\beta}=\alpha\beta\matV\matV^{\T}
\]
When $\beta=1$, and recalling that in Problem~\ref{prob:main} we
have $\matB-\matA=\alpha\matZ\matZ^{\T}$, we can take $\matV=\matZ$.
When $\beta=-1$, we obtain $\matV$ using the Sherman-Morrison-Woodbury
formula. Indeed, 
\[
\matB^{-1}-\matA^{-1}=-\alpha\matA^{-1}\matZ(\matI+\alpha\matZ^{\T}\matA^{-1}\matZ)^{-1}\matZ^{\T}\matA^{-1}
\]
so we take $\matV=\matA^{-1}\matZ(\matI+\alpha\matZ^{\T}\matA^{-1}\matZ)^{-\nicehalf}$
(note that if $\alpha=-1$ the condition that $\matA-\matZ\matZ^{\T}$
is positive definite ensures that $\matI+\alpha\matZ^{\T}\matA^{-1}\matZ$
is positive definite and the inverse square root exists).

We now have the equation 
\[
\matA^{\nicebetahalf}\Delta+\Delta\matA^{\nicebetahalf}+\Delta^{2}=\alpha\beta\matV\matV^{\T}
\]
We do an additional change of variables to make the right-hand side
positive definite even if $\alpha\neq\beta$. Let $\matC=\alpha\beta\Delta$
(so $\matB^{1/2}=\matA^{1/2}+\alpha\beta\matC$ since $\alpha=\pm1$
and $\beta=\pm1$). By multiplying the last equation on both sides
by $\alpha\beta$ we obtain the equation:
\begin{equation}
\matA^{\nicebetahalf}\matC+\matC\matA^{\nicebetahalf}+\alpha\beta\matC^{2}=\matV\matV^{\T}\label{eq:ricatti-update}
\end{equation}
Except $\matC$, all other quantities of the last equation are known,
and solving Eq.~(\ref{eq:ricatti-update}) using a low-rank positive
semidefinite $\matC$ of the form $\matC=\matU\matU^{\T}$ forms the
basis of our algorithm (see next section). However, all our algorithms
solve Eq.~(\ref{eq:ricatti-update}) approximately, since they output
low-rank solutions, while the exact solution tends to be full-rank. 

We now analyze how errors in solving Eq.~(\ref{eq:ricatti-update})
translate to errors in approximating $\matB^{\beta/2}$. First, let
us define the residual of an approximate solution:
\[
\matR(\tilde{\matC})\coloneqq\matV\matV^{\T}-\matA^{\nicebetahalf}\tilde{\matC}-\tilde{\matC}\matA^{\nicebetahalf}-\alpha\beta\tilde{\matC}^{2}.
\]
We start with a backward error bound, i.e. showing that if the residual
has a small norm, then $\matA^{\nicebetahalf}+\alpha\beta\tilde{\matC}$
is the square root of a matrix that is close to $(\matA+\alpha\matZ\matZ^{\T})^{\beta}$.
\begin{lem}
\label{lem:backward}We have 
\[
\FNorm{(\matA+\alpha\matZ\matZ^{\T})^{\beta}-(\matA^{\nicebetahalf}+\alpha\beta\tilde{\matC})^{2}}=\FNorm{\matR(\tilde{\matC})}
\]
\end{lem}

\begin{proof}
We have defined $\matV$ so that $(\matA+\alpha\matZ\matZ^{\T})^{\beta}=\matA^{\beta}+\alpha\beta\matV\matV^{\T}$,
so

\begin{eqnarray*}
\FNorm{(\matA+\alpha\matZ\matZ^{\T})^{\beta}-(\matA^{\nicebetahalf}+\alpha\beta\tilde{\matC})^{2}} & = & \FNorm{\matA^{\beta}+\alpha\beta\matV\matV^{\T}-(\matA^{\nicebetahalf}+\alpha\beta\tilde{\matC})^{2}}\\
 & = & \FNorm{\matA^{\beta}+\alpha\beta\matV\matV^{\T}-\matA^{\beta}-\alpha\beta\matA^{\beta/2}\tilde{\matC}-\alpha\beta\tilde{\matC}\matA^{\beta/2}-\tilde{\matC}^{2}}\\
 & = & \FNorm{\alpha\beta\matR(\tilde{\matC})}\\
 & = & \FNorm{\matR(\tilde{\matC})}
\end{eqnarray*}
\end{proof}
In order to get a bound on the forward error in terms of the backward
error, we need the following perturbation bound for the matrix square
root:
\begin{lem}
\label{lem:2.2}(\citet[Lemma 2.2]{Schmitt92}) Suppose that $\RE{\matA_{j}}\succeq\mu_{j}^{2}\matI$
, $\mu_{j}>0$, $j=1,2$. Then, both $\matA_{1}$ and $\matA_{2}$
have square roots satisfying $\RE{\matA_{j}^{1/2}}\succeq\mu_{j}\matI$
for j=1,2, and 
\[
\TNorm{\matA_{2}^{\nicehalf}-\matA_{1}^{\nicehalf}}\leq\frac{1}{\mu_{1}+\mu_{2}}\TNorm{\matA_{2}-\matA_{1}}
\]
\end{lem}

\begin{lem}
\label{lem:forward}Let $\matW$ and $\matH$ be two symmetric positive
definite matrices. The following bounds hold:{\footnotesize{}
\[
\FNorm{\matW^{\nicebetahalf}-\matH}\leq(n^{\nicehalf}\FNorm{\matW^{\beta}-\matH^{2}})^{\nicehalf}
\]
\[
\TNorm{\matW^{\nicebetahalf}-\matH}\leq\min\left(\frac{\TNorm{\matW^{\beta}-\matH^{2}}}{\sqrt{\lambda_{\min}(\matW^{\beta})}},(n^{\nicehalf}\FNorm{\matW^{\beta}-\matH^{2}})^{\nicehalf}\right)
\]
}{\footnotesize\par}
\end{lem}

\begin{proof}
The bound $\TNorm{\matW^{\nicebetahalf}-\matH}\leq\TNorm{\matW^{\beta}-\matH^{2}}/\sqrt{\lambda_{\min}(\matW^{\beta})}$
follows immediately from Lemma~\ref{lem:2.2}. The Frobenius norm
bound follows from Wihler inequality the $p$th root of positive semidefinite
matrices: for any two $n\times n$ positive semidefinite matrix $\matX$
and $\matY$ and $p>1$ we have $\|\matX^{\nicefrac{1}{p}}-\matY^{\nicefrac{1}{p}}\|_{F}^{p}\leq n^{\nicefrac{(p-1)}{2}}\|\matX-\matY\|_{F}$~\citep{Wihler09}.
We apply this inequality to $\matX=\matW^{\beta}$ and $\matY=\matH^{2}$.
\end{proof}
Finally, we obtain the following bound:
\begin{cor}
\label{cor:err-bound}Suppose that both $\matA$ and $\matA+\alpha\matZ\matZ^{\T}$
are positive definite where $\alpha=\pm1$. Let $\tilde{\matC}$ be
a positive semidefinite matrix. Assume that $\matA^{\nicebetahalf}+\alpha\beta\tilde{\matC}$
is positive definite. The following bounds holds:

\[
\FNorm{(\matA+\alpha\matZ\matZ^{\T})^{\nicebetahalf}-(\matA^{\nicebetahalf}+\alpha\beta\tilde{\matC})}\leq\left(n^{1/2}\FNorm{\matR(\tilde{\matC})}\right)^{1/2}
\]
\[
\TNorm{(\matA+\alpha\matZ\matZ^{\T})^{\nicebetahalf}-(\matA^{\nicebetahalf}+\alpha\beta\tilde{\matC})}\leq\min\left\{ \frac{\FNorm{\matR(\tilde{\matC})}}{\sqrt{\lambda_{\min}((\matA+\alpha\matZ\matZ^{\T})^{\beta})}},\left(n^{\nicehalf}\FNorm{\matR(\tilde{\matC})}\right)^{\nicehalf}\right\} 
\]
\end{cor}

\begin{prop}
Suppose that both $\matA$ and $\matA+\alpha\matZ\matZ^{\T}$ are
positive definite where $\alpha=\pm1,\beta=\pm1$. The matrix $\alpha\beta((\matA+\alpha\matZ\matZ^{\T})^{\nicebetahalf}-\matA^{\nicebetahalf})$
is a solution to Eq.~(\ref{eq:ricatti-update}). Conversely, if $\matC$
is a positive definite solution of Eq.~(\ref{eq:ricatti-update})
for which $\matA^{\nicebetahalf}+\alpha\beta\matC$ is positive definite
as well, then $\matA^{\nicebetahalf}+\alpha\beta\matC=(\matA+\alpha\matZ\matZ^{\T})^{\nicebetahalf}$.
\end{prop}

\begin{proof}
Let $\matC=\alpha\beta((\matA+\alpha\matZ\matZ^{\T})^{\nicebetahalf}-\matA^{\nicebetahalf})$.
By substation we have $\matA^{\nicebetahalf}\matC+\matC\matA^{\nicebetahalf}+\alpha\beta\matC^{2}=\alpha\beta(\matA+\alpha\matZ\matZ^{\T})^{\beta}-\alpha\beta\matA^{\beta}$.
Recall that we defined $\matV$ such that $(\matA+\alpha\matZ\matZ^{\T})^{\beta}=\matA^{\beta}+\alpha\beta\matV\matV^{\T}$
so we find that $\matA^{\nicebetahalf}\matC+\matC\matA^{\nicebetahalf}+\alpha\beta\matC^{2}=\matV\matV^{\T}$
and Eq.~(\ref{eq:ricatti-update}) holds.

Conversely, if $\matC$ is a positive definite solution of Eq.~(\ref{eq:ricatti-update})
then $\matR(\matC)=0$. Since $\matA^{\nicebetahalf}+\alpha\beta\matC$
is positive definite, Corollary~\ref{cor:err-bound} ensures that
$\FNorm{(\matA+\alpha\matZ\matZ^{\T})^{\nicebetahalf}-(\matA^{\nicebetahalf}+\alpha\beta\tilde{\matC})}\leq0$.
This can only happen if $(\matA+\alpha\matZ\matZ^{\T})^{\nicebetahalf}-(\matA^{\nicebetahalf}+\alpha\beta\tilde{\matC})=0$,
i.e. $\matA^{\nicebetahalf}+\alpha\beta\matC=(\matA+\alpha\matZ\matZ^{\T})^{\nicebetahalf}$.
\end{proof}

\section{\label{sec:algorithms}Algorithms}

In this section we describe our algorithms for solving Problem~\ref{prob:main}.
As alluded earlier, our algorithms are based on using a Riccati low-rank
solver as encapsulated by $\textrm{RicattiLRSolver}$. However, in
some combinations of $\alpha$ and \textbf{$\beta$ }there is a challenge:
while the returned $\tilde{\matC}$ is guaranteed to be positive definite,
there is no guarantee that $\matA^{\nicebetahalf}+\alpha\beta\tilde{\matC}$
is positive definite as well. Such guarantee is necessary in order
to apply Corollary~\ref{cor:err-bound}. Thus, we split our algorithm
to various cases based on the combination of $\alpha$ and $\beta$.
Our proposed algorithms are summarized in pseudo-code form in Algorithm~\ref{alg:main}.

\begin{algorithm}
\begin{algorithmic}[1]

\STATE \textbf{Inputs: $\matA^{\nicehalf}\in\R^{n\times n}$} and/or
\textbf{$\matA^{-\nicehalf}\in\R^{n\times n}$} implicitly,\textbf{
$\alpha=\pm1,\beta=\pm1$},\textbf{ $\matZ\in\R^{n\times k}$}, target
rank $r$.

\STATE \textbf{Output: $\matU\in\R^{n\times r}$ }such that\\
 $(\matA+\alpha\matZ\matZ^{\T})^{\nicebetahalf}\approx\matA^{\nicebetahalf}+\alpha\beta\matU\matU^{\T}$.

\STATE 

\STATE \uline{\mbox{$\alpha=+1,\beta=+1$}}: (only $\matA^{\nicehalf}$
required)

\STATE $\matU\gets\textrm{RiccatiLRSolver(}\matA^{\nicehalf},\matZ^{\T},+1,r)$

\STATE 

\STATE \uline{\mbox{$\alpha=-1,\beta=-1$}}: (only $\matA^{-\nicehalf}$
required)

\STATE $\matG\gets\matA^{-\nicehalf}\matZ$

\STATE Verify $\matI_{k}-\matG^{\T}\matG$ is positive definite (o/w
return error)

\STATE $\matV\gets\matA^{-\nicehalf}\matG(\matI_{k}-\matG^{\T}\matG)^{-\nicehalf}$

\STATE $\matU\gets\textrm{RiccatiLRSolver}(\matA^{-\nicehalf},\matV^{\T},+1,r)$

\STATE 

\STATE \uline{\mbox{$\alpha=-1,\beta=+1$}}:

\STATE Execute the $\alpha=-1,\beta=-1$ case to obtain $\matU_{1}$.

\STATE $\matU\gets\matA^{\nicehalf}\matU_{1}(\matI_{r}+\matU_{1}^{\T}\matA^{\nicehalf}\matU_{1})^{-\nicehalf}$

\STATE 

\STATE \uline{\mbox{$\alpha=+1,\beta=-1$}}:

\STATE Execute the $\alpha=+1,\beta=+1$ case to obtain $\matU_{1}$.

\STATE $\matU\gets\matA^{-\nicehalf}\matU_{1}(\matI_{r}+\matU_{1}^{\T}\matA^{-\nicehalf}\matU_{1})^{-\nicehalf}$

\end{algorithmic}

\caption{\label{alg:main}Algorithms for updating/downdating square root and
inverse square root.}
\end{algorithm}

\subsection{Updating ($\alpha=1$) the square root ($\beta=1$)}

This is the simplest case: we simply call $\matU\gets\textrm{RiccatiLRSolver}(\matA^{\nicehalf},\matZ^{\T},+1,r)$
and return $\matU$.

\subsection{Downdating ($\alpha=-1)$ the inverse square root ($\beta=-1$)}

We first compute $\matV=\matA^{-1}\matZ(\matI_{k}-\matZ^{\T}\matA^{-1}\matZ)^{-\nicehalf}$.
Along the way we can verify that $\matI-\matZ^{\T}\matA^{-1}\matZ$
is positive definite, which is required for $\matA-\matZ\matZ^{\T}$
to be positive semidefinite, and our algorithm to work. We now call
$\matU\gets\textrm{RiccatiLRSolver}(\matA^{-\nicehalf},\matV^{\T},+1,r)$
and return $\matU$.

\subsection{Downdating ($\alpha=-1$) the square root ($\beta=1$)}

Seemingly, we could simply call $\matU\gets\textrm{RiccatiLRSolver}(\matA^{\nicehalf},\matZ^{\T},-1,r)$
and return $\matU$. However, there is no guarantee that $\matA^{\nicehalf}-\matU\matU^{\T}$
is positive definite, and Corollary~\ref{cor:err-bound} no longer
guarantees that we have an approximation to the principal square root.

If we want to approximate the principal square root, we can first
solve for downdating the inverse square root ($\alpha=-1,\beta=-1$),
obtaining $\matU_{1}$ such that $(\matA-\matZ\matZ^{\T})^{-\nicehalf}\approx\matA^{-\nicehalf}+\matU_{1}\matU_{1}^{\T}$.
We now use the Sherman-Morrison-Woodbury formula to note that 
\[
(\matA^{-\nicehalf}+\matU_{1}\matU_{1}^{\T})^{-1}=\matA^{\nicehalf}-\matA^{\nicehalf}\matU_{1}(\matI_{r}+\matU_{1}^{\T}\matA^{\nicehalf}\matU_{1})^{-1}\matU_{1}^{\T}\matA^{\nicehalf}
\]
so we return $\matU=\matA^{\nicehalf}\matU_{1}(\matI_{r}+\matU_{1}^{\T}\matA^{\nicehalf}\matU_{1})^{-\nicehalf}$.

\subsection{Updating ($\alpha=1$) the inverse square root ($\beta=-1$)}

Again, calling the Riccati solver directly might return a corrected
matrix which it not neccessarily positive definite, and it will not
be a good approximation to the principal square root. To approximate
the principal square root, we first solve the updating problem for
the square root $(\matA+\matZ\matZ^{\T})^{\nicehalf}\approx\matA^{\nicehalf}+\matU_{1}\matU_{1}^{\T}$
($\alpha=+1,\beta=+1$), and then use the Sherman-Morrison-Woodbury
formula to find a $\matU$ such that $(\matA^{\nicehalf}+\matU_{1}\matU_{1}^{\T})^{-1}=\matA^{-\nicehalf}-\matU\matU^{\T}$.
We omit the details and simply refer the reader to the pseudo code
description in Algorithm~\ref{alg:main}.

\subsection{Costs}

The main cost of the algorithms is in solving the Riccati equation.
Even when the Sherman-Morrison-Woodbury formula is needed to ensure
positive definiteness of the correction, its cost of $O(nk^{2})$
is subsumed by the cost of solving the Riccati equation. Overall,
under our assumptions on the cost of solving the Riccati equation,
the overall cost of the algorithms is $O((T_{\matA^{\nicehalf}}+T_{\matA^{-\nicehalf}})r^{2}+nr^{4})$
where $T_{\matA^{\nicehalf}}$ and $T_{\matA^{-\nicehalf}}$ are the
costs of the taking products of $\matA^{\nicehalf}$ and $\matA^{-\nicehalf}$
(respectively) with a vector. In many of the applications we discuss
in the next section $\matA$ is diagonal, in which case the cost of
the algorithms reduces to $O(nr^{4})$.

\section{\label{sec:applications}Applications}

\subsection{ZCA Whitening of High Dimensional Data}

\emph{Whitening transformations} are designed to transform a random
vector (or samples of that random vector) with a known covariance
matrix into a new random vector whose covariance is the identity matrix.
Suppose that $\x\in\R^{p}$ is a random vector with covariance matrix
$\Sigma\in\R^{p\times p}$. Let $\matW$ be any matrix such that $\matW^{\T}\matW=\Sigma^{-1}$;
such a matrix is called a \emph{whitening matrix}. Then the covariance
matrix of the random vector $\z=\matW\x$ is the identity, so the
random variable has been \emph{whitened}. There are several possible
choices for $\matW$, leading to different whitening transformations.
ZCA whitening is the whitening transformation defined by $\matW=\Sigma^{-\nicehalf}$.

In practice, the ZCA transformation is learned from data. Given samples
$\x_{1},\dots,\x_{n}$ an estimate $\hat{\Sigma}$ of $\Sigma$ is
formed, and $\hat{\matW}=\hat{\Sigma}^{-\nicehalf}$ is used for the
ZCA whitening matrix. A common choice is to use the sample covariance
matrix $\matS_{n}=n^{-1}\matX_{c}^{\T}\matX_{c}$ for $\hat{\Sigma}$
where $\matX_{c}$ is the data matrix whose rows are $\x_{1},\dots,$$\x_{n}$
after centering (subtraction of the mean).

However, it is well appreciated in the statistical literature that
when the random vectors are high dimensional, i.e. when $p$ is of
the same order as $n$ (or much larger), then the sample covariance
$\matS_{n}$ is a poor estimate of $\Sigma$. Indeed, one can easily
see that if $p>n$ then $\matS_{n}$ is not even invertible so the
ZCA transformation is not even defined. In high dimensional settings
it is common to adopt the \emph{spiked covariance model }of~\citet{Johnstone01}.
In the spiked covariance model it is assumed that the covariance matrix
has the following form:
\begin{equation}
\Sigma=\sigma^{2}\matI_{p}+\matZ\matZ^{\T}\label{eq:spiked}
\end{equation}
 for some $\matZ\in\R^{p\times k}$ ($k$ is a parameter).

Suppose that we have formed an estimate $\hat{\Sigma}$ of $\Sigma$
with the same structure as in Eq.~(\ref{eq:spiked}), and we want
to transform the samples using ZCA whitening. Explicitly computing
the square root of $\hat{\Sigma}$ requires $O(p^{3})$, which is
prohibitive when $p$ is large, and additional $O(p^{2}n)$ is required
for applying the transformation to the data. Instead, we can use our
algorithm to find a $\matU\in\R^{n\times r}$ with $r=O(k)$ such
that 
\[
\hat{\Sigma}^{-\nicehalf}\approx\sigma^{-1}\matI-\matU\matU^{\T}
\]
We can then apply the ZCA transformation in $O(npk)$.

\subsection{Updating/Downdating Polar Decomposition and ZCA Transformed Data}

Given a matrix $\matX\in\mathbb{R}^{n\times d}$ where $n\geq d$,
a polar decomposition of it is 
\[
\matX=\matU\matP
\]
where $\matU$ has orthonormal columns and $\matP$ is symmetric positive
semidefinite. The matrix $\matP$ is always unique, and is given by
$\matP=(\matX^{\T}\matX)^{\nicehalf}$. If $\matX$ has full rank,
then $\matP$ is positive definite, and $\matU=\matX\matP^{-1}$.
The polar decomposition can be computed using a reduced SVD, so the
cost of computing a polar decomposition is $O(nd^{2})$. There are
quite a few uses for the polar decomposition \citep{Higham86}.

We now consider the following updating/downdating problem. Let us
denote the rows of $\matX$ by $\x_{1},\dots,\x_{n}\in\R^{d}$, i.e.
row $j$ of $\matX$ is $\x_{j}^{\T}$. Suppose we already have a
polar decomposition $\matX=\matU\matP$ of $\matX$. The \emph{downdating
problem} is: compute the polar decomposition of a matrix $\matX_{-}$
obtained by removing one row from $\matX$. The\emph{ updating problem}
is: compute the polar decomposition of $\matX_{+}$, a matrix obtained
by adding a single row to $\matX$.

We describe an algorithm for downdating a polar decomposition. The
updating algorithm is almost the same. Without loss of generality,
assume we remove the last row from $\matX$:
\[
\matX_{-}^{\T}\matX_{-}=\matX^{\T}\matX-\x_{n}\x_{n}^{\T}\,.
\]
This is a rank-1 perturbation of $\matX^{\T}\matX$. The $\matP$-factor
of $\matX_{-}$, which we denote by $\matP_{-}$, is obtained by computing
the square root of $\matX_{-}^{\T}\matX_{-}$, and we already have
the square root $\matP$ for the unperturbed matrix $\matX^{\T}\matX$.
So we can use the algorithm described in Section~\ref{sec:algorithms}
to find a matrix $\matU\in\R^{d\times k}$ for some small $r$ (a
parameter; e.g., $k=4)$ such that 
\[
\matP_{-}\approx\tilde{\matP}_{-}\coloneqq\matP-\matU\matU^{\T}
\]
Assume now that $\matX_{-}$ is full rank as well. We now have 
\[
\matU_{-}\approx\tilde{\matU}_{-}\coloneqq\matX_{-}\tilde{\matP}_{-}^{-1}
\]
Using the Sherman-Morrison-Woodbury formula: 
\[
\tilde{\matP}_{-}^{-1}=\matP^{-1}+\matP^{-1}\matU(\matI_{r}-\matU^{\T}\matP^{-1}\matU)^{-1}\matU^{\T}\matP^{-1}
\]
so
\[
\tilde{\matU}_{-}=\matX_{-}\matP^{-1}(\matI_{r}+\matU(\matI_{k}-\matU^{\T}\matP^{-1}\matU)^{-1}\matU^{\T}\matP^{-1})
\]
Now notice that $\matX_{-}\matP^{-1}$ is just the first $n-1$ rows
of $\matX\matP^{-1}=\matU$ so we do not need to recompute it. Multiplying
$\matX_{-}\matP^{-1}$ by $(\matI_{d}+\matU(\matI_{r}-\matU^{\T}\matP^{-1}\matU)^{-1}\matU^{\T}\matP^{-1})$
can be done, utilizing the low rank structure of that matrix, using
$O(ndr)$ operations. If $r\ll n$ this is a big reduction in complexity
over $O(nd^{2})$.

The polar decomposition is closely connected to ZCA whitening. Suppose
that $\matX$ is a data matrix whose rows are sampled from a zero
mean random vector (or, alternatively, $\matX$ has been centered).
The $\matU$-factor is equal, up to scaling, to the ZCA transformed
data, while the inverse of the $\matP$-factor is, up to scaling,
the ZCA whitening matrix itself. Using the ability to update the inverse
square root, the procedure for updating/downdating the polar decomposition
can be adjusted to update/downdate ZCA. Updating/downdating ZCA can
be useful if you want to transform data that arrives over time while
the covariance matrix itself changes slowly. That is, data point $\x_{j}$
is sampled with covariance $\Sigma_{j}$. If we assume the covariance
changes slowly, we can keep an approximate ZCA of the data by taking
the sample covariance over a sliding window. To do so efficiently,
we can use the proposed ZCA updating/downdating procedure to first
remove outdated data (a downdate operation), and then add the newly
arrived data (and update operation).

\subsection{Sampling from a Multivariate Normal Distribution with Perturbed Precision
Matrix}

Consider a random vector $\x\in\R^{n}$ following a multivariate normal
distribution with precision matrix $\matQ$, i.e. $\x\sim N(\mu,\matQ^{-1}).$
Suppose we want to sample $\x$. This can be accomplished by sampling
a vector $\z$ from the standard multivariate normal distribution
(i.e. $\z\sim N(0,\matI_{n})$) and then computing the sample $\x=\mu+\matQ^{-\nicehalf}\z$.

In certain cases the matrix $\matQ$ has the structure of a low-rank
perturbation of a fixed precision matrix, i.e. $\matQ=\matQ_{0}+\matZ\matZ^{\T}$.
Assuming we already have computed the inverse square root of $\matQ_{0}$,
we can use our algorithms to compute a $\matU$ such that $\matQ^{-\nicehalf}\approx\matQ_{0}^{-\nicehalf}-\matU\matU^{\T}$.
We can then sample efficiently from $\x$.

Such cases can occur in Gibbs Samplers for Bayesian inference on spatially
structured data. An example is the image reconstruction task discussed
in \citep{Bardsley12}. The computational bottleneck in the algorithm
proposed in \citep{Bardsley12} is sampling from a conditional Gaussian
distribution whose precision matrix has the structure $\matQ=\gamma_{\textrm{prior}}\matL+\gamma_{\textrm{obs}}\matZ\matZ^{\T}$
where $\matL$ is a fixed discrete Laplace operator that encodes prior
smoothness assumptions on the image, while $\matZ^{\T}$ encodes how
the high-resolution images are blurred and downsampled to yield low-resolution
images.

\subsection{\label{subsec:shampoo_app}Preconditioned Second-Order Optimization}

Recently introduced by \citet{GKS18}, \noun{Shampoo} is a preconditioned
second-order optimization method for solving problems in which the
parameter space is naturally organized as a $m\times n$ matrix or
higher order tensor. Here, we consider the matrix-shaped case. In
this case, at the core, \noun{Shampoo} performs update steps of the
form 
\begin{equation}
\matW_{t+1}\gets\matW_{t}-\eta\matL_{t}^{-\nicefrac{1}{4}}\matG_{t}\matR_{t}^{-\nicefrac{1}{4}}\label{eq:shampoo-update}
\end{equation}
where $\eta$ is the learning rate, $\{\matW_{t}\}$ are the parameters
at time $t$, $\{\matG_{t}\}$ are the gradients at time $t$, and
\[
\matL_{t}\coloneqq\epsilon\matI_{m}+\sum_{s=1}^{t}\matG_{s}\matG_{s}^{\T}\quad\matR_{t}\coloneqq\epsilon\matI_{n}+\sum_{s=1}^{t}\matG_{s}^{\T}\matG_{s}.
\]

If $m\gg nt$ then it is better to avoid holding $\matL_{t}^{-\nicefrac{1}{4}}$
explicitly (which is $m\times m$ ), and simply hold an implicit representation
of both $\matL_{t}^{-\nicefrac{1}{4}}$ and $\matL_{t}^{-\nicehalf}$
as diagonal plus low-rank matrices. In each iteration we can update
both by applying our algorithm twice. Since $n\ll m$, we can store
$\matR_{t}$ explicitly, and compute $\matR_{t}^{-\nicefrac{1}{4}}$
in each iteration. If $n\gg mt$ we can reverse roles, implicitly
keeping $\matR_{t}$ and explicitly keeping $\matL_{t}$. Even if
both $m$ and $n$ are of comparable size, then in some cases $\matG_{t}$
is of low-rank, and again we can track $\matL_{t}$ and $\matR_{t}$
using our algorithm.

We stress that our method has an advantage over~\citep{FHL22} when
applied to SHAMPOO. \citet{FHL22} can only handle perturbations of
the identity, so when applied to compute the square root of $\matL_{t+1}$
it cannot use the square root of $\matL_{t}$ (which is available
from the previous iteration). Our algorithm, on the other hand, can
use the fact that $\matL_{t+1}=\matL_{t}+\matG_{t}\matG_{t}^{\T}$
for a low rank update of the previous iteration. If the matrices are
explicitly held, \citet{FHL22} will have a cost per iteration that
grows linearly with iteration count, while with our algorithm the
cost will stay constant.

\subsection{Faster Generalized Least Squares with a Spiked Weight Matrix}

Let $\matX\in\R^{n\times d}$ and $\y\in\R^{n}$. In Generalized Least
Squares (GLS) we wish to find the minimizer
\begin{equation}
\min_{\w\in\R^{d}}\WNorm{\matX\w-\b}\label{eq:gls}
\end{equation}
where $\matW\in\R^{n\times n}$ is some symmetric positive definite
weight matrix, and $\WNorm{\z}\coloneqq\sqrt{\z^{\T}\matW\z}$. In
this section, we focus on the cases that $\matW$ can be written as
a diagonal plus a definite low rank perturbation $\matW=\matD+\alpha\matZ\matZ^{\T}$
where $\matD\in\R^{n\times n}$ is diagonal, and $\matZ\in\R^{n\times k}$.

A statistical motivation for this problem is generalized linear regression
with a spiked covariance matrix. Assume that the rows of $\matX$
correspond to data points $\x_{1},\dots,\x_{n}$, the entries $y_{1},\dots,y_{n}$
of $\y$ are responses. We now assume that the responses follow the
model $y_{i}=\x_{i}^{\T}\w^{\star}+\epsilon_{i}$ where the vector
of noise elements $\epsilon_{1},\dots,\epsilon_{m}$ is distributed
according to ${\cal N}(0,\matC)$ for some covariance matrix $\matC=\matD+\alpha\matZ\matZ^{\T}$.
The optimal unbiased estimate of $\w^{\star}$ is obtained by solving
Eq.~(\ref{eq:gls}) with $\matW=\matC^{-1}$.

One can easily see that Eq.~(\ref{eq:gls}) is equivalent to 
\begin{equation}
\min_{\w\in\R^{d}}\TNorm{\matW^{\nicehalf}\matX\w-\text{\ensuremath{\matW}}^{\nicehalf}\y}\label{eq:ls-converted}
\end{equation}
Once we have efficiently computed $\matW^{\nicehalf}\matX$ and $\matW^{\nicehalf}\y$,
we can leverage faster, sketching based, least squares algorithms~\citep{Woodruff14,DrineasMahoney16}.
Using the algorithms from Section~\ref{sec:algorithms} we can compute
a $\matU\in\R^{n\times r}$ such that $\matW^{\nicehalf}=\matD^{-\nicehalf}-\alpha\matU\matU^{\T}$
with $r=O(k).$ We can them compute $\matW^{\nicehalf}\matX$ and
$\matW^{\nicehalf}\y$ efficiently.

\section{Experiments}

We report experiments exploring the ability of our algorithm to find
low-rank corrections to the square root or inverse square root of
a perturbed matrix. In our experiments, we focus on the quality of
the corrections found. We do not report running time since the code
we used for the algebraic Riccati solver (downloaded from the homepage
of~\citet{BV14}) is not optimized to take advantage of the structures
present in the input matrices for the specific algebraic Riccati equations
our algorithm solves.

\subsection{\label{subsec:syn}Synthetic Experiments}

We first test our algorithm on randomly generated matrices, and compare
them to the approximations obtained using the algorithm of~\citet{BKS18},
the approximation obtained using the algorithm of~\citet{BCKS20},
and the optimal correction obtained by zeroing out the smallest eigenvalues
of the exact correction. The matrix $\matA\in\R^{100\times100}$ is
a diagonal matrix, whose diagonal is either sampled uniformly from
$U(0,1)$, or whose entries are logarithmically spaced between $10^{-3}$
and $10^{3}$. The perturbation is $\matD=\z\z^{\T}$, where $\z$
is a normalized Gaussian vector. We consider both updates ($\alpha=+1$)
and downdates $(\alpha=-1$). In case of downdates, we multiply $\z$
by $0.1$ to ensure positive definiteness after the downdate. We consider
both the square root ($\beta=+1)$ and inverse square root ($\beta=-1$).
We plot the relative error $\FNorm{(\matA+\alpha\z\z^{\T})^{\nicebetahalf}-(\matA^{\nicebetahalf}+\alpha\beta\matU\matU^{\T})}/\FNorm{(\matA+\alpha\z\z^{\T})^{\nicebetahalf}}$
as a function of the rank of the update (\#columns in $\matU$).

We consider the algorithm of~\citet{BKS18} applied in two different
ways. The first, which is labeled in the graphs as ``Krylov Method'',
simply runs the algorithm of algorithm of~\citet{BKS18} for $r$
iterations ($r$ is the target rank), to obtain a rank $r$ perturbation.
In the second, which is labeled in the graphs as ``Truncated $r^{2}$
Krylov Method'', runs the algorithm of algorithm of~\citet{BKS18}
for $r^{2}$ iterations, but then computes the best rank $r$ approximation
to the correction (which is of rank $r^{2}$). This algorithm has
the same asymptotic cost for diagonal matrices as our algorithm.

We implemented the algorithm of~\citet{BCKS20} using the Rational
Krylov Toolbox for MATLAB\footnote{\url{http://guettel.com/rktoolbox/guide/html/index.html}}.
The poles are obtained using that toolbox as well. The algorithm is
labeled in the graphs as ``Rational Krylov Method''.

\begin{figure}[h]
\begin{centering}
\begin{tabular}{cccc}
\includegraphics[width=0.2\columnwidth]{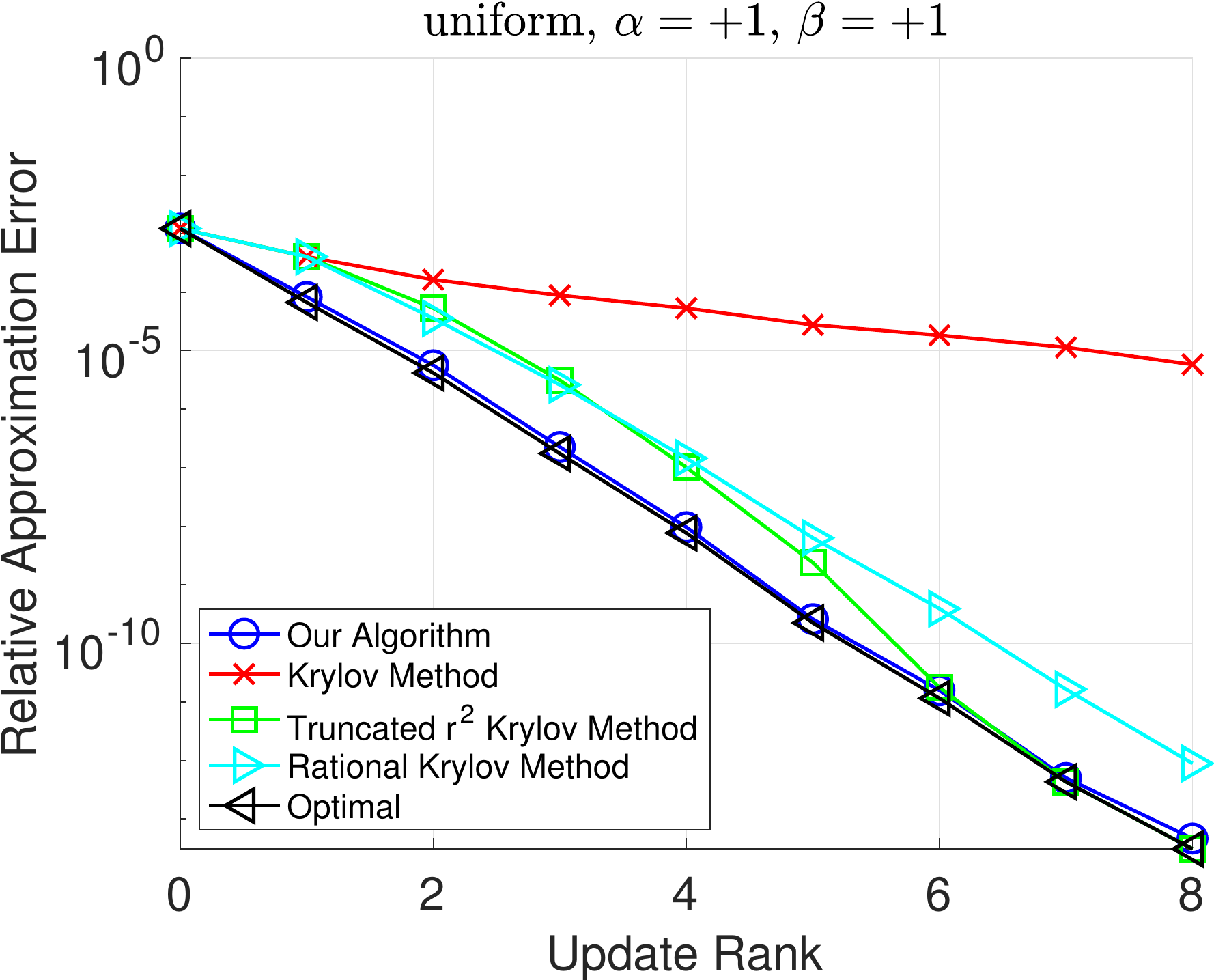} & \includegraphics[width=0.2\columnwidth]{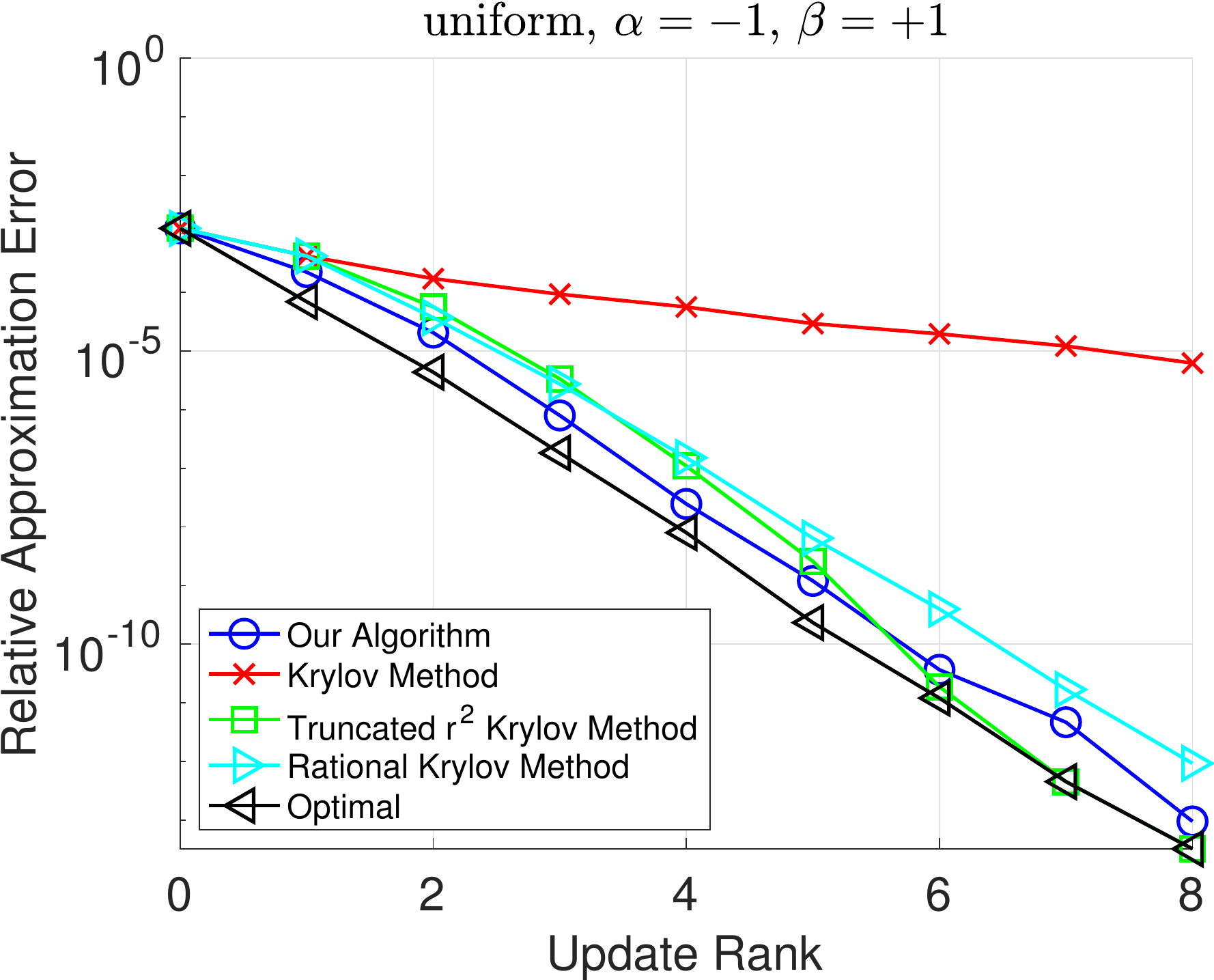} & \includegraphics[width=0.2\columnwidth]{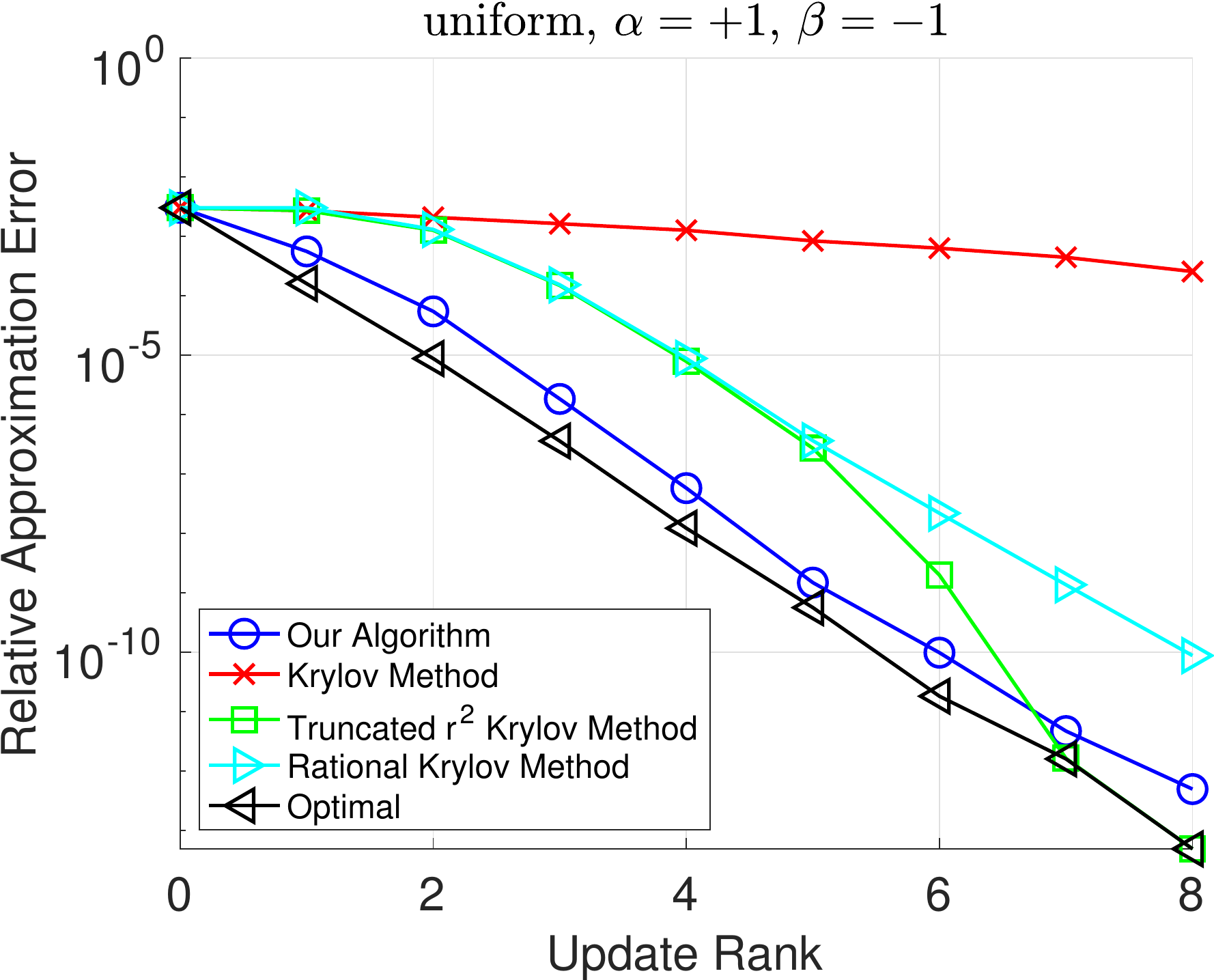} & \includegraphics[width=0.2\columnwidth]{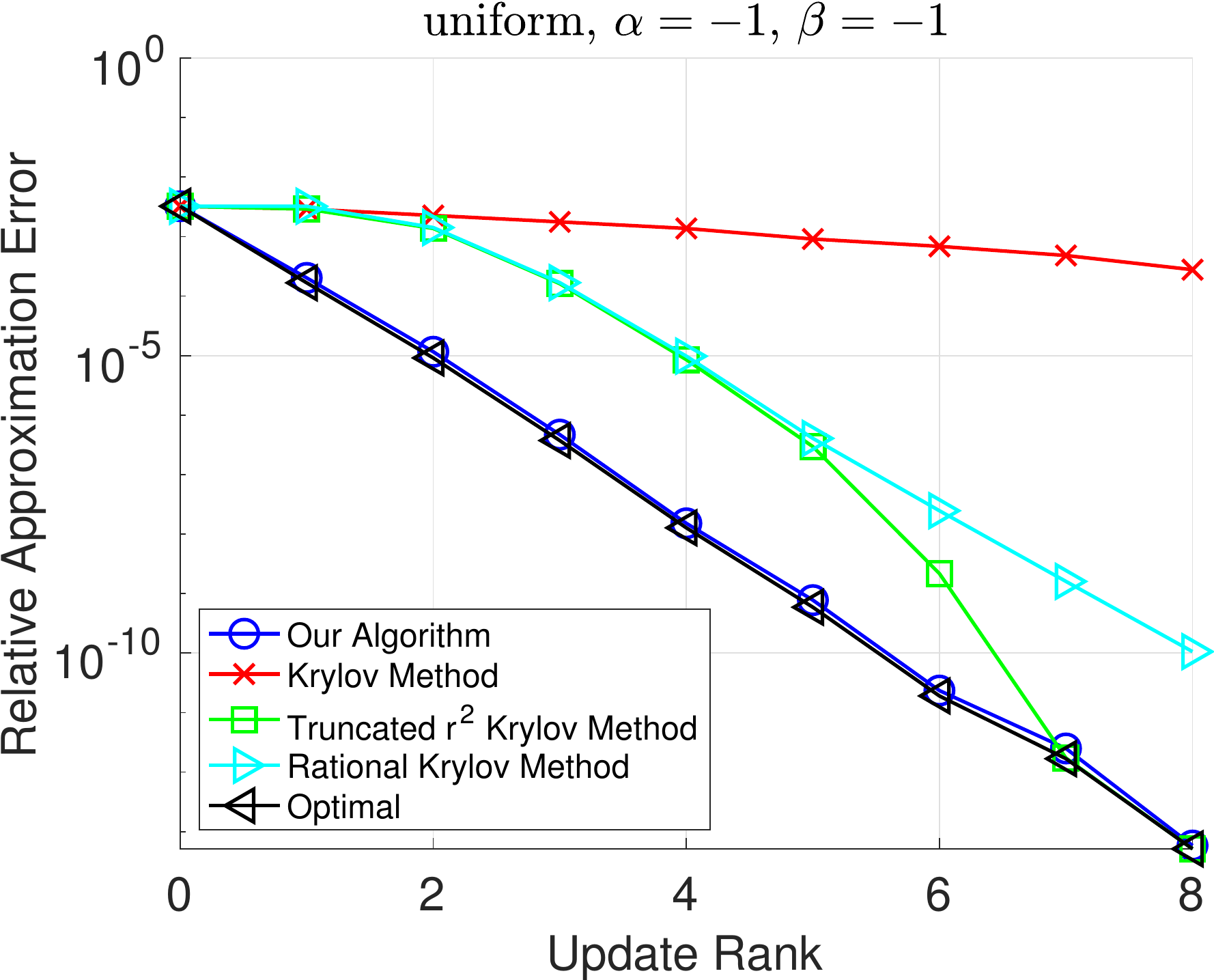}\tabularnewline
\end{tabular}
\par\end{centering}
\caption{\label{fig:app_syn_experiment_uniform}Comparison of our algorithm
to approximations obtained using~\citet{BKS18} and to the optimal
approximation, on randomly generated matrices of the form $\protect\matD+\protect\z\protect\z^{\protect\T}$
where $\protect\matD$ is diagonal with uniformly sampled entries
from $N(0,1)$.}
\end{figure}

\begin{figure}[h]
\begin{centering}
\begin{tabular}{cccc}
\includegraphics[width=0.2\columnwidth]{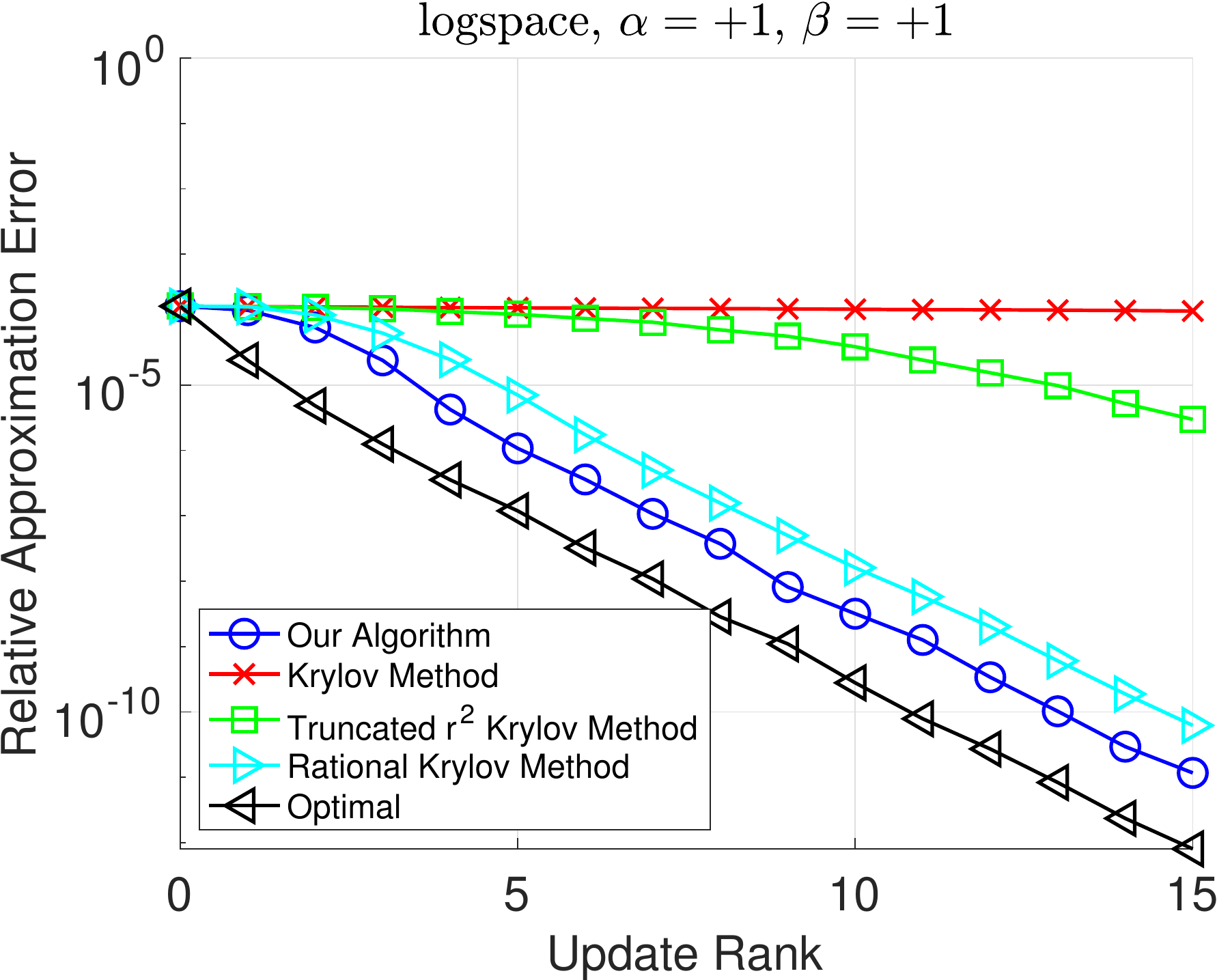} & \includegraphics[width=0.2\columnwidth]{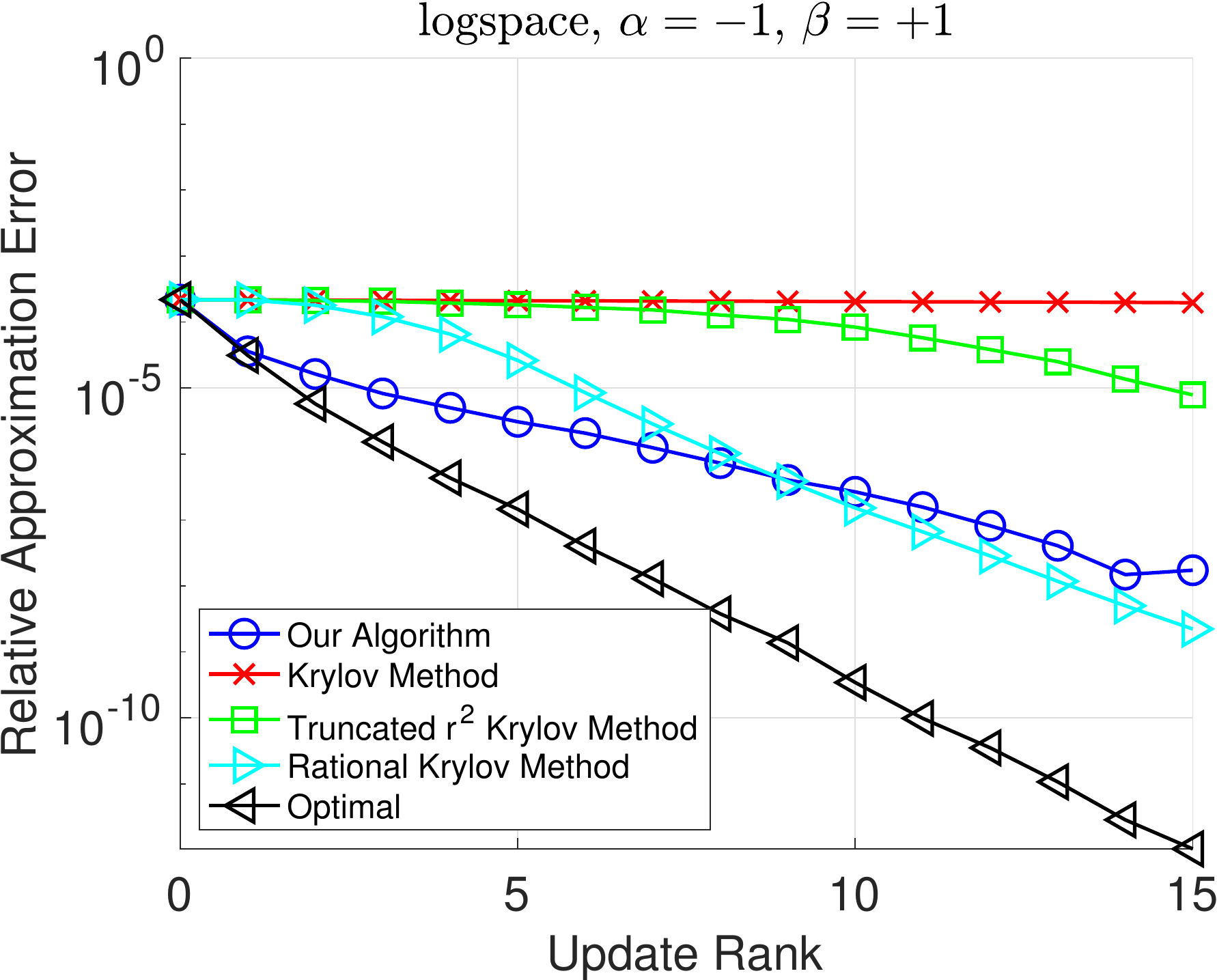} & \includegraphics[width=0.2\columnwidth]{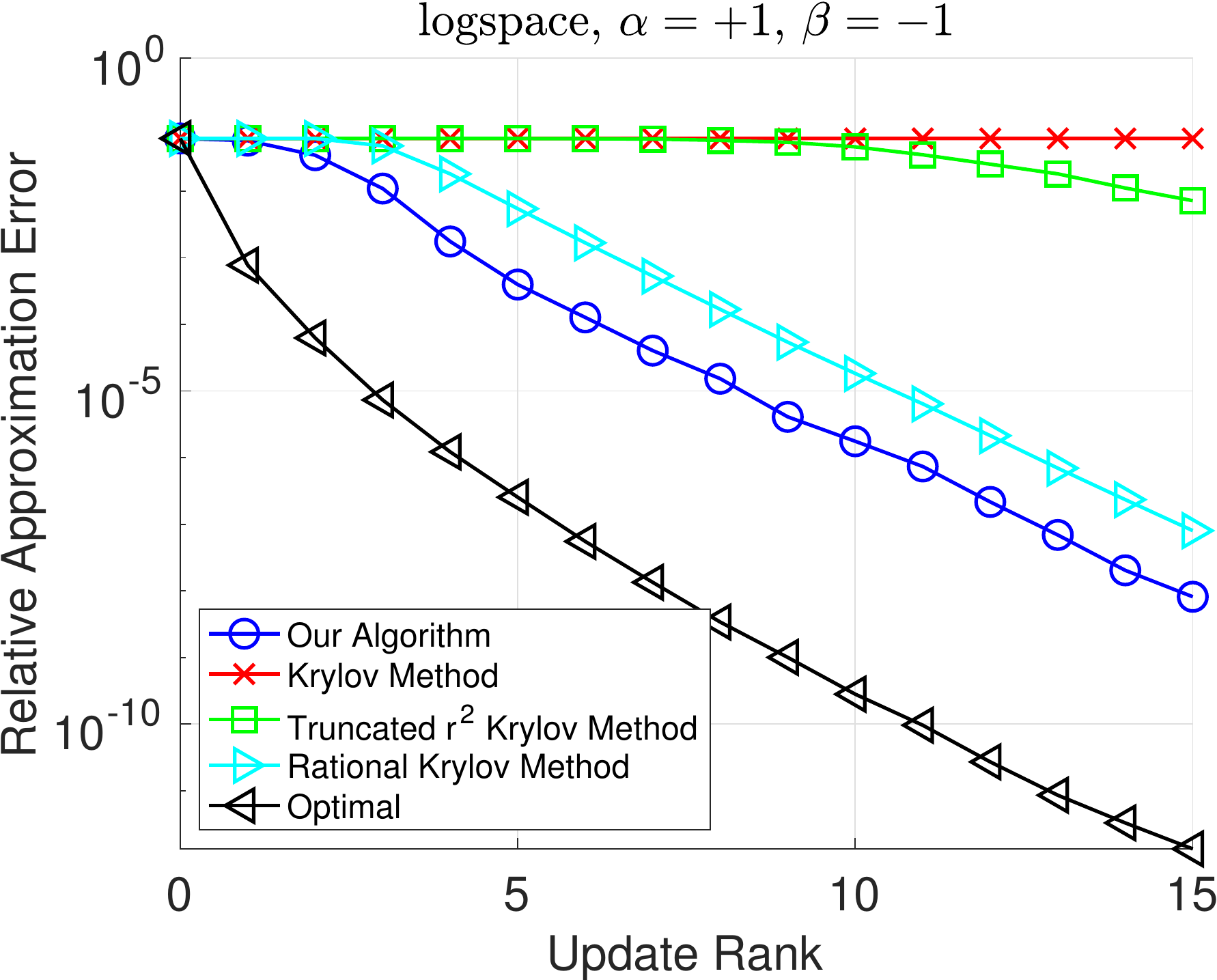} & \includegraphics[width=0.2\columnwidth]{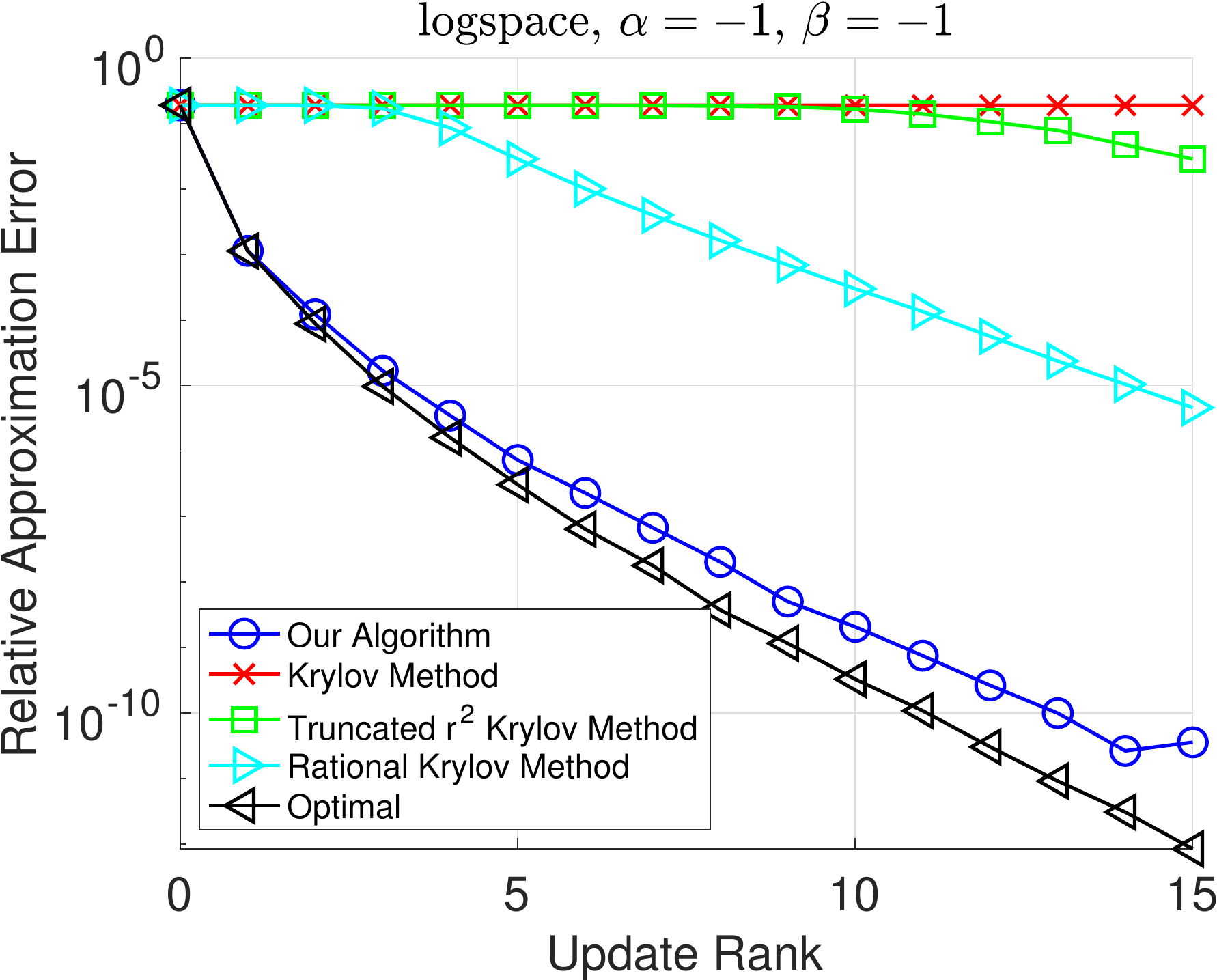}\tabularnewline
\end{tabular}
\par\end{centering}
\caption{\label{fig:app_syn_experiment_logspace}Comparison of our algorithm
to approximations obtained using~\citet{BKS18} and to the optimal
approximation, on randomly generated matrices of the form $\protect\matD+\protect\z\protect\z^{\protect\T}$
where $\protect\matD$ is diagonal with logarithmically spaced entries.}
\end{figure}

Figures~\ref{fig:app_syn_experiment_uniform} and~\ref{fig:app_syn_experiment_logspace}
show the result for all different combinations. Our algorithm is clearly
able to find much better approximations than the Krylov method of~\citet{BKS18}.

\subsection{Matrices Arising from Second-Order Optimization}

\begin{figure}
\begin{centering}
\begin{tabular}{cc}
\includegraphics[width=0.4\columnwidth]{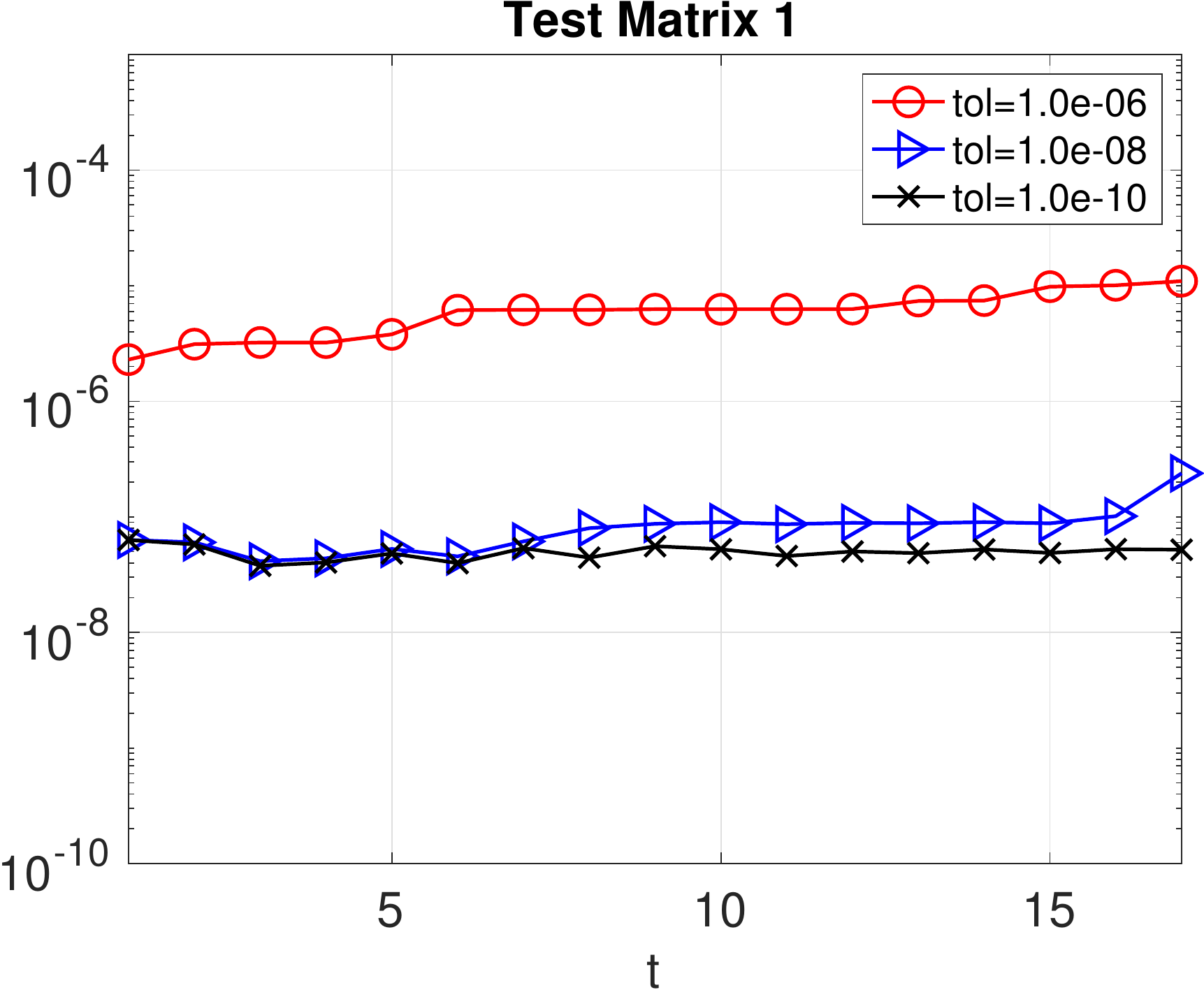} & \includegraphics[width=0.4\columnwidth]{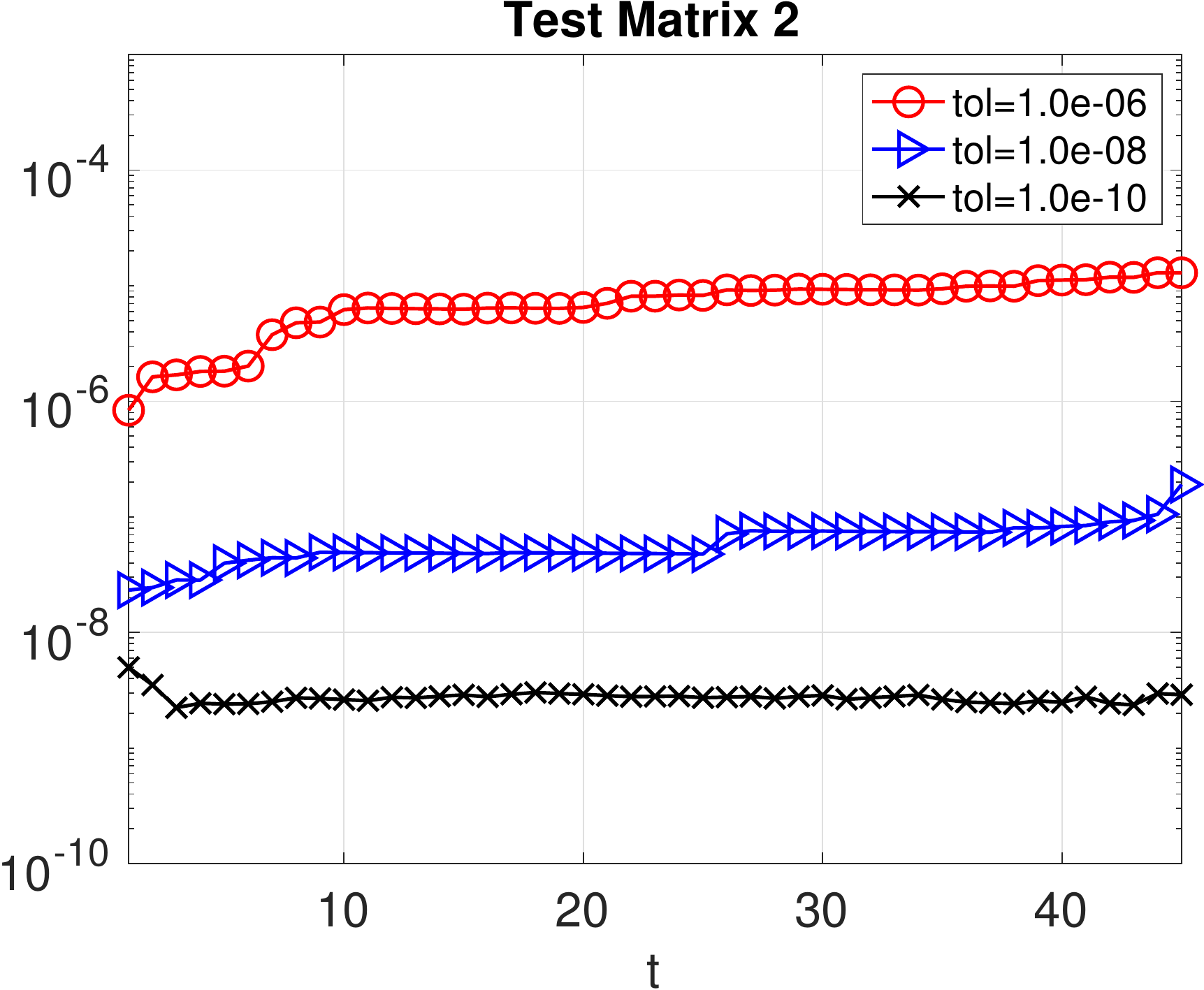}\tabularnewline
\end{tabular}
\par\end{centering}
\caption{\label{fig:shampoo_experiment}Simulating the use of our algorithm
to track $\protect\matL_{t}^{-\nicefrac{1}{4}}$ or $\protect\matR_{t}^{-\nicefrac{1}{4}}$
in \noun{Shampoo, }on two test matrices from the \noun{Lingvo} framework.}
\end{figure}

Our next set of experiments simulates the use of our algorithm to
track $\matL_{t}^{-\nicefrac{1}{4}}$ or $\matR_{t}^{-\nicefrac{1}{4}}$
in \noun{Shampoo} (see Section~\ref{subsec:shampoo_app}). We obtain
and preprocess the data in a similar way to~\citep{FHL22}, but the
experiment itself is different. We downloaded two test matrices available
from the \noun{Lingvo} framework for \noun{TensorFlow} \citep{Lingvo}
and are available on GitHub\footnote{\href{https://github.com/tensorflow/lingvo/tree/master/lingvo/core/testdata}{https://github.com/tensorflow/lingvo/tree/master/lingvo/core/testdata}}.
These matrices are obtained by accumulating updates with $\alpha=0$.
Unfortunately, the provided test matrices are only the final accumulated
matrix, and do not contain the discrete updates themselves, so we
need to extract updates that accumulate to the final matrix. We do
so in a similar fashion to the one used by \citet{FHL22}: we compute
an eigendecomposition, and keep dominant factors that are bigger than
$0.1$. This yields a rank $82$ approximation to the first matrix,
and a rank $221$ approximation to the second matrix.

For the experiment, we split the low rank approximation into discrete
updates of rank $5$. So we now have a sequence of $\{\matG_{s}\}$,
each $\matG_{s}$ having 5 columns. Our goal is to efficiently track
$\matL_{t}^{-\nicefrac{1}{4}}$ where $\matL_{t}=\alpha\matI+\sum_{s=1}^{t}\matG_{s}\matG_{s}^{\T}=\matL_{t-1}+\matG_{t}\matG_{t}^{\T}$
where we set $\alpha=0.001$. We use our algorithm to form two sets
of updates, $\{\matU_{s}\}$ and $\{\matW_{s}\}$, each with $5$
columns, such that $\matL_{t}^{-\nicefrac{1}{2}}\approx\matL_{t-1}^{-\nicefrac{1}{2}}+\matU_{t}\matU_{t}^{\T}\approx\alpha^{-\nicehalf}\matI+\sum_{s=1}^{t}\matU_{s}\matU_{s}^{\T}$
and $\matL_{t}^{-\nicefrac{1}{4}}\approx\matL_{t-1}^{-\nicefrac{1}{4}}+\matW_{t}\matW_{t}^{\T}\approx\alpha^{-\nicefrac{1}{4}}\matI+\sum_{s=1}^{t}\matW_{s}\matW_{s}^{\T}$.
Note that in iteration $t$, we consider the latest perturbation to
be of $\matL_{t-1}$ and the approximation of $\matL_{t-1}^{-\nicefrac{1}{2}}$.
This saves time (since the perturbation rank does not grow) and storage
(we do not need to keep previous updates). We plot in Figure~\ref{fig:shampoo_experiment}
the distance between $\matL_{t}^{-\nicefrac{1}{4}}$ and its approximation,
as it evolves over time. We do so for three different tolerances in
the internal Riccati low-rank solver. We see that our algorithm is
able to track $\matL_{t}^{-\nicefrac{1}{4}}$ well over time, without
the errors blowing-up.

\subsection*{Acknowledgments.}

The authors thank the anonymous reviewers for their helpful comments.
Haim Avron and Shany Shmueli were partially supported by the Israel
Science Foundation (grant no. 1272/17) and by the US-Israel Binational
Science Foundation (grant no. 2017698). Petros Drineas was partially
supported by NSF 10001415 and NSF 10001390.\bibliographystyle{icml2021}
\bibliography{bibtex}

\appendix

\section{\label{app:negative-ricatti}Solving Eq.~(\ref{eq:algebraic-ricatti})
for $\alpha=-1$}

If $\alpha=+1$, Eq.~(\ref{eq:algebraic-ricatti}) is an instance
of the \emph{algebraic Riccati equation}, for which \citet{BV14}
proposed an algorithm for finding an approximate low rank solution.
That algorithm can be adjusted to the $\alpha=-1$, by making several
small adjustments to the various \emph{Euclidean }components (the
Riemannian ones are obtained by converting the Euclidean components
to Riemannian ones). We frame the expressions with an $\alpha=\pm1$
to cover both cases concurrently.
\begin{itemize}
\item Optimization problem: the new optimization problem is:
\begin{equation}
\min_{\rank\matX=k,\Delta\succeq0}0.25\FNormS{\matA^{\T}\matX+\matX\matA^{\T}+\alpha\X\matB\matB^{\T}\X^{\T}-\matU\matU^{\T}}\label{eq:are_minus_alpha}
\end{equation}
\item Gradient expression: The method in \citep{BV14} keeps $\matX$ in
factorized low-rank form $\matX=\matY\matY^{\T}$. Let 
\[
\matS(\matY)\coloneqq\matA^{\T}\matY\matY^{\T}+\matY\matY^{\T}\matA^{\T}+\alpha\matY\matY^{\T}\matB\matB^{\T}\matY\matY^{\T}-\matU\matU^{\T}.
\]
The cost function in Eq.~(\ref{eq:are_minus_alpha}) is $F(\matY)\coloneqq0.25\FNormS{\matS(\matY)}$.
Simple calculations show that the gradient of $\matF$ is: 
\[
\nabla F(\matY)=\matA\matS(\matY)\matY+\matS(\matY)\matA^{\T}\matY+\alpha\matS(\matY)\matY^{\T}\matB\matB^{\T}\matY+\alpha\matB\matB^{\T}\matY\matY^{\T}\matS(\matY)
\]
\item Hessian expression: using a similar technique as \citep{BV14}, we
calculate directional derivative in direction $\w$ by computing $\lim\frac{1}{\epsilon}\left(\nabla F(\matY+\epsilon\matW)-\nabla F(\matY)\right)$.
In the limit, $\matS(\matY)$ and $\matS(\matY+\epsilon\matW)$ include
$\alpha$ and so are slightly different from the ones used in \citep{BV14}.
Nevertheless, the expression for the Euclidean Hessian is:
\begin{align*}
\matA\matS(\matY)\matY+\matA\matS(\matY)\matW+\matS(\matY)\matA^{\T}\matW+\matS(\matY)\matA^{\T}\matY+\\
\alpha\left(\matS(\matY)\matY\matY^{\T}\matB\matB^{\T}\matY+\matS(\matY)\matW^{\T}\matB\matB^{\T}\matY+\matS(\matY)\matW\matY^{\T}\matB\matB^{\T}\matY+\matS(\matY)\matY\matY^{\T}\matB\matB^{\T}\matW\right)+\\
\alpha\left(\matB\matB^{\T}\matY\matY^{\T}\matS(\matY)\matW+\matB\matB^{\T}\matY\matY^{\T}\matS(\matY)\matY+\matB\matB^{\T}\matY\matW^{\T}\matS(\matY)\matY+\matB\matB^{\T}\matW\matY^{\T}\matS(\matY)\matY\right)\\
\end{align*}
\end{itemize}

\end{document}